\documentclass{amsart}
\usepackage{latexsym,amsxtra,amscd,ifthen,amsmath,color, multicol,hyperref}
\usepackage{amsfonts}
\usepackage{verbatim}
\usepackage{amsmath}
\usepackage{amsthm}
\usepackage{amssymb}
 \usepackage[all,cmtip]{xy}
 \usepackage[normalem]{ulem} 
\usepackage{enumerate}
 \usepackage{tikz}
\usetikzlibrary{matrix,arrows}

\theoremstyle{plain}
\newtheorem{theorem}{Theorem}[section]

\newtheorem{lemma}[theorem]{Lemma}

\newtheorem{proposition}[theorem]{Proposition}
\newtheorem{hypothesis}[theorem]{Hypothesis}
\newtheorem{corollary}[theorem]{Corollary}

\theoremstyle{definition}
\newtheorem{definition}[theorem]{Definition}
\newtheorem{example}[theorem]{Example}
\newtheorem{notation}[theorem]{Notation}

\newtheorem{convention}[theorem]{Convention}

\newtheorem{remark}[theorem]{Remark}
\newtheorem{question}[theorem]{Question}

\numberwithin{equation}{section}

\makeatletter              
\let\c@equation\c@theorem  
\makeatother

\newcommand{\zg}{\gamma}
\newcommand{\zl}{\lambda}

\newcommand{\kk}{\Bbbk}

\newcommand{\uqsl}{u_q(\mathfrak{sl}_2)}
\newcommand{\uqslu}{u_q^{\geq 0}(\mathfrak{sl}_2)}
\newcommand{\uqsll}{u_q^{\leq 0}(\mathfrak{sl}_2)}

\begin{document}

\title[Pointed Hopf actions on path algebras]
{Actions of some pointed Hopf algebras on path algebras of quivers}

\author{Ryan Kinser}
\address{Department of Mathematics, University of Iowa, Iowa City, Iowa 52242, USA}
\email{ryan-kinser@uiowa.edu}

\author{Chelsea Walton}
\address{Department of Mathematics, Massachusetts Institute of Technology, Cambridge, Massachusetts 02139,
USA}
\email{notlaw@math.mit.edu}

\bibliographystyle{abbrv}       

\begin{abstract}
We classify Hopf actions of Taft algebras $T(n)$ on path algebras of quivers, in the setting where the quiver is loopless, finite, and Schurian.  As a corollary, we see that every quiver admitting a faithful $\mathbb{Z}_n$-action (by directed graph automorphisms) also admits inner faithful actions of a Taft algebra.  Several examples for actions of the Sweedler algebra $T(2)$ and for actions of $T(3)$ are presented in detail. We then extend the results on Taft algebra actions on path algebras to actions of the Frobenius-Lusztig kernel $u_q(\mathfrak{sl}_2)$, and to actions of the Drinfeld double of $T(n)$.
\end{abstract}



\maketitle

\setcounter{tocdepth}{1}
\tableofcontents

\section{Introduction} \label{sec:intro}

Let $n$ be an integer $\geq 2$ and let $\kk$ be a field containing a primitive $n$-th root of unity $\zeta$.  Both $\kk$ and $n$ will be fixed but arbitrary subject to this condition throughout the paper.  Note that if char$(\kk) =p > 0$, this implies that $p$ and $n$ are coprime.
All algebras in this work are associative $\kk$-algebras and let an unadorned $\otimes$  denote $\otimes_{\kk}$.

Generalizing the classical notion of a group acting on an algebra by automorphisms, one can consider actions of Hopf algebras (e.g., quantum groups).
However, one obstacle is that the intricate structure of a Hopf algebra often prevents non-trivial actions on an algebra. When such actions exist, they can be difficult to construct and are not generally well understood. This paper presents a case where a classification of these actions is achieved. Here, we consider actions of some finite dimensional, pointed Hopf algebras, namely actions of Taft algebras as a start. The algebra being acted upon is the path algebra of a quiver, and actions are subject to Hypothesis~\ref{hyp:standing}.  All necessary background, including definitions, is recalled in Section~\ref{sec:background}.
In particular, we address the following question:
\begin{question} \label{question}
When does the path algebra of a quiver admit a non-trivial action of a (finite dimensional, pointed) Hopf algebra?  Specifically, of a Taft algebra?
\end{question}

Actions by Taft algebras are referred to as {\it Taft actions} for short.
We give a complete answer to the question above for Taft actions, and extend Taft actions to actions of the quantum group $\uqsl$ and actions of the Drinfeld double of a Taft algebra, under the following conditions. 
\begin{hypothesis} \label{hyp:standing} Unless stated otherwise, we impose the assumptions below.
\begin{enumerate}
\item The quiver $Q$ is finite, loopless, and Schurian.
\item Hopf actions preserve the path length filtration of the path algebra $\kk Q$. 
\end{enumerate}
\end{hypothesis}

It is easy to see that $Q$ must at least admit a non-trivial action of the cyclic group $\mathbb{Z}_n$ (namely, the group of grouplike elements of $T(n)$) to admit a non-trivial action of the $n$-th Taft algebra $T(n)$; see Example~\ref{ex:GfixesQ0}. So, loosely speaking, we are interested in when a path algebra of a quiver that admits classical symmetry admits additional quantum symmetry. 
Our strategy is to identify a class of quivers which is small enough so that we can explicitly describe all Taft actions on their path algebras, but large enough so that every quiver admitting a Taft action is a union of quivers in this class. We call these quivers the {\it $\mathbb{Z}_n$-minimal quivers} [Definition~\ref{def:typeAB}].  The reader may wish to look over Section \ref{sec:Sweedler} early on for a complete account of the case $n=2$: actions of the Sweedler algebra $T(2)$ on $\mathbb{Z}_2$-minimal quivers. 

To begin, we first note that any action on a path algebra must restrict to an action on the subalgebra generated by the vertices, by Hypothesis \ref{hyp:standing}(b).  So we start by classifying Taft actions on products of fields in Proposition \ref{prop:TaftactQ0}.  
Then, the form of actions on vertices places significant restrictions on actions on the arrows.  The following theorem summarizes our results, with reference to more detailed statements in the body of the paper.   Here, we let $g$ and $x$  be the standard generators of $T(n)$, where $g$ is grouplike, $x$ is $(1,g)$ skew-primitive, and $xg = \zeta gx$ for 
$\zeta$ some primitive $n$-th root of unity (see Section \ref{subsec:Hopfactions}).  We identify the cyclic group generated by $g$ with $\mathbb{Z}_n$.

\begin{theorem} \label{thm:mainintro}
Let $Q$ be a quiver, and suppose we have a Taft action on its path algebra $\kk Q$.

\noindent (a)  The Taft action determines an action of $\mathbb{Z}_n$ on $Q$ by quiver automorphisms [Lemma~\ref{lem:GactkQ}].

\noindent (b) Each $\mathbb{Z}_n$-orbit of vertices is stable under $T(n)$.
If we let $\{e_1, \dotsc, e_m\}$ be the collection of trivial paths corresponding to some orbit of vertices, numbered so that $g \cdot e_i = e_{i+1}$ with subscripts taken modulo $m$, then the action of $x$ on these 
is given by
$$x \cdot e_i = \gamma \zeta^i e_i - \gamma \zeta^{i+1} e_{i+1},$$
for any scalar $\gamma \in \kk$ [Proposition~\ref{prop:TaftactQ0}].

\noindent (c) For each arrow $a$ of $Q$, the action of $x$ on $a$ is given by
$$x \cdot a = \alpha a + \beta (g \cdot a) + \lambda \sigma(a),$$
for some scalars $\alpha, \beta$, and $\lambda$.  Here,  $\sigma(a)$ is an arrow or trivial path with the same source as $a$ and the same target as $g\cdot a$ [Notation~\ref{not:sigma}, Proposition~\ref{prop:Taftactarrow}].
Furthermore, when $Q$ is a $\mathbb{Z}_n$-minimal quiver, these scalars are determined explicitly by the formulae \eqref{eq:A} and \eqref{eq:B} [Theorems~\ref{thm:T(n)typeA},~\ref{thm:T(n)typeB}]. 
\end{theorem}

With an explicit parametrization of Taft actions on path algebras of $\mathbb{Z}_n$-minimal quivers, it remains to show that this is sufficient to parametrize Taft actions on path algebras of quivers subject to Hypothesis~\ref{hyp:standing}.  To do this, we introduce the notion of a $\mathbb{Z}_n$-component of a quiver with $\mathbb{Z}_n$-action [Definition \ref{def:component}].
These are the smallest subquivers of $Q$ which have at least one arrow and are guaranteed to be stable under the action of $T(n)$ for any choice of parameters.
Moreover, see Definition~\ref{def:compatible} for the notion of a {\it compatible} collection of Taft actions.

\begin{theorem}[Lemmas~\ref{lem:Gminl}, \ref{lem:restrict}, Theorem~\ref{thm:glue}, Corollary~\ref{cor:non-trivial}] \label{thm:introglue} 
Fix an action of  $\mathbb{Z}_n$ on a quiver $Q$. Then, $Q$ decomposes uniquely into a union of its $\mathbb{Z}_n$ components, and any Taft action on $\kk Q$ restricts to an action on each component.
Moreover, this decomposition gives a bijection between Taft actions on $\kk Q$ and compatible collections of Taft actions on the $\mathbb{Z}_n$-components of $Q$.
In particular, any path algebra of a quiver with a faithful action of $\mathbb{Z}_n$ admits an inner faithful action of the $n$-th Taft algebra $T(n)$.
\end{theorem}

As mentioned above, we  extend these results to get actions of other finite dimensional, pointed Hopf actions on path algebras of quivers.

\begin{theorem}[Theorem~\ref{thm:Uqslaction} and~\ref{thm:Daction}, Section~\ref{sec:glueUqD}] Fix an action of  $\mathbb{Z}_n$ on a quiver $Q$. Let $q \in \kk$ be a $2n$-th root of unity. Additional restraints on parameters are determined so that the Taft actions on the path algebra of $Q$ produced in Theorems~\ref{thm:mainintro} and~\ref{thm:introglue} extend to an action of the Frobenius-Lusztig kernel $\uqsl$ and to an action of the Drinfeld double  of $T(n)$.
\end{theorem}

As a consequence of the theorem above, we obtain that path algebras of quivers, that  admit $\mathbb{Z}_n$-symmetry, are {\it algebras} in the category of Yetter-Drinfeld modules over $T(n)$ by \cite{Majid:doubles}; see also \cite[Exercise~13.1.6]{MR2894855}.  Hence, motivated by the process of bosonization, or Radford's biproduct construction to produce (potentially new) Hopf algebras (c.f \cite{Majid} \cite[Theorems~11.6.7 and~11.6.9]{MR2894855}), we pose the following question.

\begin{question}
Let $Q$ be a quiver that admits $\mathbb{Z}_n$-symmetry. When does the path algebra $\kk Q$ admit the structure of a {\it Hopf algebra} in the category of Yetter-Drinfeld modules over $T(n)$?
\end{question}

\subsection{Comparisons to other work}

A path algebra $\kk Q$ is naturally a coalgebra, where the comultiplication of a path is the sum of all splits of the path. There are previous studies on extending the coalgebra structure on $\kk Q$ to a graded Hopf algebra, most notably Cibils and Rosso's work on Hopf quivers \cite{CibilsRosso0, CibilsRosso1}. Here, when $\kk Q$ admits the structure of a Hopf algebra, the group of grouplike elements of $\kk Q$ consists of the vertex set $Q_0$ of $Q$. Moreover, any arrow $a \in Q_1$ is a skew-primitive element as $\Delta(a) = s(a) \otimes a + a \otimes t(a)$. One example of their theory is a construction of $T(n)$ from a Hopf quiver, and in this case it has the regular action on the path algebra of this quiver.  
Our study produces many more examples of Taft actions on path algebras, as our construction allows for non-trivial actions on path algebras of  {\it any} quiver that admits $\mathbb{Z}_n$-symmetry.

There is also some intersection of our work with a recent preprint of Gordienko \cite{Gordienko}. On the one hand, he works in the setting of Taft actions on arbitrary finite dimensional algebras, whereas path algebras of quivers are not always finite dimensional.  For example, in \cite[Theorem~1]{Gordienko} Gordienko classifies Taft algebra actions on products of matrix algebras, while our Proposition~\ref{prop:TaftactQ0} only classifies Taft algebra actions on products of fields (equivalently, path algebras of arrowless quivers). On the other hand, Gordienko's classification in \cite[Theorem 3]{Gordienko} is restricted to actions giving $T(n)$-\emph{simple} module-algebras, whereas we have classified all Taft actions on path algebras (subject to Hypothesis~\ref{hyp:standing}). With the exception of special parameter values, the path algebras in this work are not simple with respect to the Taft algebra action: one can easily see from our explicit formulas that the Jacobson radical (the ideal generated by the arrows of $Q$) is typically a non-trivial two-sided $T(n)$-invariant ideal.

There is an abundance of literature on both the study of quantum symmetry of graphs and group actions on directed graphs from the viewpoint of operator algebras, including \cite{Banica, BBC, BPW, Bichon, KumjianPask}. Connections to our results merit further investigation.

Other works investigating relations between path algebras of quivers and Hopf algebras can be found in the following references 
\cite{MR2047446, MR2609178,  MR2741251, MR2089252, MR2267572}.

Moreover in \cite{MS}, Montgomery-Schneider provide similar results for actions of Taft algebras, and extended actions of  $\uqsl$ and of $D(T(n))$,  on the commutative algebras: $\kk(u)$, and $\kk[u]/(u^n-\beta)$ with $\beta \in \kk$.

\section{Background} \label{sec:background}

We begin by defining Taft algebras and Hopf algebra actions. We then discuss path algebras of quivers, which will be acted on by Taft algebras throughout this work.

\subsection{Taft algebras and Hopf algebra actions} \label{subsec:Hopfactions} 
Let $H$ be a Hopf algebra with coproduct $\Delta$, counit $\varepsilon$, and antipode $S$. 
A nonzero element $g \in H$ is {\it grouplike} if $\Delta(g) = g \otimes g$, and the set of grouplike elements of $H$ is denoted by $G(H)$. This forces $\varepsilon(g) = 1$ and $S(g) = g^{-1}$. An element $x \in H$ is {\it $(g,g')$-skew-primitive}, for grouplike elements $g, g'$ of $H$, when $\Delta(x) = g \otimes x + x \otimes g'$. In this case, $\varepsilon(x) = 0$ and $S(x) = -g^{-1}xg'^{-1}$. 
The following examples of Hopf algebras will be used throughout this work. 

\begin{definition}[Taft algebra $T(n)$, Sweedler algebra $T(2)$] \label{def:Taft} 
The {\it Taft algebra} $T(n)$ is a $n^2$-dimensional Hopf algebra generated by a grouplike element $g$ and a $(1,g)$-skew-primitive element $x$, subject to relations: 
\[g^n =1,\quad x^n = 0, \quad xg = \zeta gx\] 
for $\zeta$ a primitive $n$-th root of unity. The 4-dimensional Taft algebra $T(2)$ is known as the {\it Sweedler algebra}.
\end{definition}

\noindent Note that $G(T(n))$ is isomorphic to the cyclic group $\mathbb{Z}_n$, generated by $g$.

We now recall basic facts about Hopf algebra actions; refer to
\cite{Montgomery} for further details.  
A left $H$-module $M$ has left $H$-action structure map denoted by $\cdot : H \otimes M \rightarrow M$.  We use Sweedler notation $\Delta(h) = \sum h_1 \otimes h_2$ for coproducts.

\begin{definition}[$H$-action] \label{def:Hopfact} Given a Hopf algebra $H$ and an algebra $A$, we say
that {\it $H$ acts on $A$} (from the left) if
\begin{enumerate}
\item $A$ is a left $H$-module,
\item $h \cdot (pq) = \sum (h_1 \cdot p)(h_2 \cdot q)$, and 
\item $h \cdot 1_A = \varepsilon(h) 1_A$ 
\end{enumerate}
for all $h \in H$, and $p,q \in A$.  In this case, we say that $A$ is a {\it left $H$-module algebra}. 
Equivalently, the multiplication map $\mu_A\colon A\otimes A \to A$ and unit map $\eta_A\colon \kk \to A$ are morphisms of $H$-modules, so $A$ is an algebra in the monoidal category of left $H$-modules.
\end{definition}

For the Taft actions in this work, consider the following terminology.

\begin{definition}[Extending a $G$-action] Given an action of a group $G$ on an algebra $A$, we say that an action of a Hopf algebra $H$ on $A$ {\it extends the $G$-action on $A$} if the restriction of the $H$-action to $G(H)$ agrees with the $G$-action via some isomorphism $G(H) \simeq G$.
\end{definition}

In this paper, we are interested in the case where $G=\mathbb{Z}_n$ and $H=T(n)$ in the above definition.  Moreover, it is useful to restrict to $H$-actions that do not factor through proper quotient Hopf algebras.  

\begin{definition}[Inner faithful]
A module $M$ over a Hopf algebra $H$ is \emph{inner faithful} if the action of $H$ on $M$ does not factor through a quotient Hopf algebra of $H$; that is, $IM \neq 0$ for any nonzero Hopf ideal $I \subset H$.
A Hopf action of $H$ on an algebra $A$ is inner faithful is $A$ is inner faithful as an $H$-module.
\end{definition}

The following lemma is likely known to experts, but does not seem to be readily accessible in the literature, so we provide a proof.

\begin{lemma}\label{lem:Taftfaithful}
Every nonzero bi-ideal of $T(n)$ contains $x$.  Therefore, a Taft action on an algebra $A$ is inner faithful if and only if $x \cdot A \neq 0$.
\end{lemma}
\begin{proof}
Writing $H:=T(n)$, since $H \cong H^*$ as Hopf algebras it suffices to prove the dual statement.  Namely, since $x$ generates the radical of $H$, the dual approach is to show that every proper sub-bialgebra of $H$ is contained in the coradical of $H$.

Suppose that $A \subseteq H$ is a nonzero Hopf sub-bialgebra of $H$ which is not contained in the coradical of $H$.  We will show that $A=H$.
Since the coradical $H_0$ of $H$ is the span of the grouplike elements $\{g^i \mid i=0, \dotsc, n-1\}$, we have that $A$ contains a nonzero element  $f =  hx^j +$ (terms of lower $x$-degree), where $j \geq 1$ and $h \in H_0$. Say, $h = \sum_{d=0}^{i} \nu_d g^d$, for $\nu_d \in \kk$ with $\nu_i \neq 0$. Since $A$ inherits the coproduct from $H$,
$$\Delta(\nu_i g^i x^j) = \sum_{\ell = 0}^j \left[\begin{array}{c}j \\ \ell\end{array}\right]_{\zeta} \nu_i g^i x^{j-\ell} \otimes \nu_i g^{j-\ell+i}x^\ell,$$
which is in $A \otimes A$. Here, the equality above holds by \cite[Lemma~7.3.1]{MR2894855}. By the maximality of $j$ and $i$ and by taking $\ell =0$ above, we have $f \in A$ implies that $g^i x^j \in A$.  Now applying $\Delta$ to $g^i x^j$ yields $g^i x^{j-\ell} \in A$ for $\ell = 0, \dots, j$. So, we get $g^ix \in A$. Likewise, apply $\Delta$ to $g^i x$ to conclude that $g^i, g^{i+1} \in A$. Since $g$ has finite multiplicative order, $g^{-i} \in A$ as well. Thus, both $g$ and $x$ are in $A$, so $A = H$, as desired.\end{proof}

So the extension of a faithful cyclic group ($\mathbb{Z}_n$) action on an algebra $A$ to a Taft algebra ($T(n)$) action is inner faithful if and only if $x \cdot A \neq 0$.
To study module algebras of $T(n)$, the following standard fact will also be of use.

\begin{lemma} \label{lem:oppositeaction} For each $T(n)$-action on a algebra $A$, there is a natural action of $T(n)$ on the opposite algebra ($A^{op}$, $\ast$), say denoted by $\diamond$, as follows:
\begin{equation} \label{eq:oppact}
g\diamond p = g^{-1} \cdot p \quad \text{and} \quad x\diamond p = g^{-1}x \cdot p.
\end{equation}
This gives a bijection between $T(n)$-actions on $A$ and on $A^{op}$.
\end{lemma}

\begin{proof}
For any Hopf algebra $H$ and algebra $A$, we get that $A$ is an $H$-module algebra if and only if ($A^{op}$, $\ast$) is an $H^{cop}$-module algebra. Here, $H^{cop}$ is the co-opposite algebra of $H$ and $\Delta^{cop} = \tau \circ \Delta$ with $\tau(\sum h_1 \otimes h_2) = \sum h_2 \otimes h_1$. Indeed, $h \cdot (pq) = h \cdot (q \ast p) = \sum (h_2 \cdot q) \ast (h_1 \cdot p) = \sum (h_1 \cdot p)(h_2 \cdot q)$. The map sending $g$ to $g^{-1}$ and $x$ to $g^{-1} x$ gives an isomorphism $T(n) \cong T(n)^{cop}$ as Hopf algebras, and the result follows.
\end{proof}

\subsection{Path algebras of quivers}\label{sec:pathalgebra}
A {\it quiver} is another name for a directed graph, in the context where the directed graph is used to define an algebra.  Here, we review the basic notions and establish notation. More detailed treatments can be found in the texts \cite{assemetal, Schiffler:2014aa}. Formally, a quiver $Q=(Q_0, Q_1, s, t)$ consists of a set of {\it vertices} $Q_0$, a set of {\it arrows} $Q_1$, and two functions $s, t\colon Q_1 \to Q_0$ giving the {\it source} and {\it target} of each arrow, visualized as
\[
s(a) \xrightarrow{\quad a \quad} t(a).
\]
A {\it path} $p$ in $Q$ is a sequence of arrows $p=a_1 a_2 \cdots a_\ell$ for which $t(a_i)=s(a_{i+1})$ for $1 \leq i \leq \ell-1$.  (Note that we read paths in left-to-right order.)  The length of $p$ is the number of arrows $\ell$.  There is also  a length 0 {\it trivial path} $e_i$ at each vertex $i \in Q_0$, with $s(e_i) = t(e_i) = i$.

A quiver $Q$ has a {\it path algebra} $\kk Q$ whose basis consists of all paths in $Q$, and multiplication of basis elements is given by composition of paths whenever it is defined, and 0 otherwise.  More explicitly, we have the following definition.

\begin{definition}[Path algebra] The {\it path algebra} $\kk Q$ of a quiver $Q$ is the $\kk$-algebra presented by generators from the set $Q_0 \sqcup Q_1$ with the following relations:
\begin{enumerate}
\item $\sum_{i \in Q_0} e_i = 1$;
\item $e_i e_j = \delta_{ij} e_i$ for all $e_i, e_j \in Q_0$; and
\item $a = e_{s(a)}a = ae_{t(a)}$ for all $a \in Q_1$.
\end{enumerate}
\end{definition}

Condition (a) is due to the assumption that $|Q_0|<\infty$. Further, $e_i$ is a primitive orthogonal idempotent in $\kk Q$ for all $i$.
So, $\kk Q$ is an associative algebra with unit, which is finite dimensional if and only if $Q$ has no path of positive length which starts and ends at the same vertex.
Notice that $\kk Q$ is naturally filtered by path length. Namely, if we let $F_i$ be the subspace of $\kk Q$ spanned by paths of length at most $i$, then we have that $F_i \cdot F_j \subseteq F_{i+j}$.

\begin{definition}[Schurian] \label{def:Schurian}
A quiver $Q$ is said to be {\it Schurian} if, given any two vertices $i$ and $j$, there is at most one arrow which starts at $i$ and ends at $j$.  
\end{definition}
Note that the definition above does not exclude oriented 2-cycles.

\begin{definition}[Covering, Gluing, $\circledast$]
A quiver $Q$ is \emph{covered by} a collection of subquivers $Q^1, \dotsc, Q^r$ if $Q = \bigcup_i Q^i$.  We say $Q$ is obtained by \emph{gluing} the collection $Q^1, \dotsc, Q^r$ if the collection covers $Q$, and in addition $Q^i \cap Q^j$ consists entirely of vertices when $i \neq j$; in this case, write
$$Q = Q^1 \circledast \cdots \circledast Q^r.$$
\end{definition}

\begin{remark} If a quiver $Q$ is obtained by gluing subquivers $Q^1, \dots, Q^r$, then we get that $\kk Q$ is the factor of the free product of path algebras $\kk Q^1 \ast \cdots \ast \kk Q^r$ by the ideal generated by $\{e_{i,v} - e_{j,v}\}$, where for each pair $(i,j)$, the index $v$ varies over the vertices of $Q^i \cap Q^j$. Here,  $e_{\ell,v}$ indicates the trivial path at $v$, for $v \in (Q^\ell)_0$.
\end{remark}

\subsection{Group actions on path algebras of quivers}

Now we consider group actions on quivers and on their path algebras.

\begin{definition}[Quiver automorphism] \label{def:quiveriso}
Let $Q, Q'$ be quivers, and consider two maps of sets $f_0 \colon Q_0 \to Q'_0$  and $f_1 \colon Q_1\to Q'_1$.  

\begin{enumerate}[(1)]
\item We say that the pair $f = (f_0, f_1)$ is a {\it quiver isomorphism} if each of these maps is bijective, and they form a commuting square with the source and target operations.  That is, for all $a \in Q_1$ we have
$$(f_0 \circ s)(a) = (s \circ f_1)(a) \quad \text{and} \quad (f_0 \circ t)(a) = (t \circ f_1)(a).$$ 
\item We say that the pair $f = (f_0, f_1)$ is a {\it quiver automorphism of $Q$} if $f$ satisfies (1) and $Q'=Q$.
\item A group $G$ is said to {\it act on} $Q$, or $Q$ is {\it $G$-stable}, if $G$ acts on the sets $Q_0$ and $Q_1$ such that each element of $G$ acts by a quiver automorphism. 
\end{enumerate}
\end{definition}

\begin{remark} A quiver automorphism induces an automorphism of its path algebra, but not every automorphism of a path algebra is of this form. See Lemma \ref{lem:GactkQ}.
\end{remark}

\begin{remark} \label{rem:Taftopp}
Given a quiver $Q$, we can form the \emph{opposite quiver} $Q^{\rm op}$ by interchanging the source and target functions $s$ and $t$.  It is clear from the definition of the path algebra that $\kk(Q^{\rm op}) \cong (\kk Q)^{\rm op}$.  Hence, Lemma~\ref{lem:oppositeaction} implies there is a bijection between $T(n)$-actions on $\kk Q$ and on $\kk Q^{\rm op}$ given by \eqref{eq:oppact}. See Remark~\ref{rem:IIItoIV} for an illustration.
\end{remark}

\section{Preliminary Results} \label{sec:prelim}

In this section, we present preliminary results on Taft actions on path algebras of loopless, Schurian quivers. We begin by studying actions on vertices, first giving a simple lemma regarding group actions on $\kk Q_0$ [Lemma~\ref{lem:GactQ0}].  Then, we extend this result to classifying Taft actions on $\kk Q_0$ [Proposition~\ref{prop:TaftactQ0}]. Preliminary results on Taft actions on path algebras $\kk Q$ are provided [Proposition~\ref{prop:Taftactarrow}], along with an example in the case where $Q_0$ is fixed by the group of grouplike elements of $T(n)$ [Examples~\ref{ex:GfixesQ0}].

The following lemma is elementary, but we provide the details in any case.

\begin{lemma} [$G$-action on $\kk Q_0$] \label{lem:GactQ0} 
Let $G$ be a group,  let $Q_0$ be a set of vertices, and let $\{e_i\}_{i \in Q_0}$ be the corresponding primitive orthogonal idempotents in $\kk Q_0$.  Then, any $G$-action on the set $Q_0$ induces an action on the ring $\kk Q_0$, given by $g \cdot e_i = e_{g\cdot i}$ for each $i \in Q_0$ and $g \in G$.
Moreover, every $G$-action on $\kk Q_0$ arises in this way.
\end{lemma}

\begin{proof} 
The first statement is clear, so suppose for the converse that we have a $G$-action on $\kk Q_0 \simeq \kk \times \kk \times \cdots \times \kk$.  To act as a ring automorphism, each element of $G$ must send a complete collection of primitive orthogonal idempotents to another such collection.  But in this case, the set $\{e_i\}_{i \in Q_0}$ is the unique such collection.
So this set must be permuted by $G$, defining an action of $G$ on the set $Q_0$.
\end{proof}

Now we turn our attention to group actions on arbitrary path algebras of quivers.

\begin{lemma}[$G$-action on $\kk Q$] \label{lem:GactkQ}
Let $G$ be a group and $Q$ a quiver, and suppose that $G$ acts by automorphisms of $\kk Q$, preserving the path length filtration.  Then, the action of each $g \in G$ on $Q$  is given by
\begin{enumerate}[(i)]
\item a quiver automorphism $\phi \colon Q \to Q$, along with
\item a collection of nonzero scalars $\mu_a \in \kk^{\times}$ indexed by the arrows $a$ of $Q$,
\end{enumerate}
such that $g\cdot a = \mu_a \phi(a)$ for each $a \in Q_1$.  (To be clear, both $\phi$ and the collection $\mu_a$ depend on $g$.)
\end{lemma}

\begin{proof}
Since $G$ preserves the path length filtration, it acts by permutations on the vertex set, by Lemma~\ref{lem:GactQ0}.
Then for each $g \in G$ and $a \in Q_1$, we have that $g\cdot a = g \cdot (e_{s(a)} a) = (g\cdot e_{s(a)})(g\cdot a),$
showing that $g\cdot a$ lies in the span of arrows starting at $g\cdot e_{s(a)}$.  Similarly, we see that $g\cdot a$ lies in the $\kk$-span of arrows with target $g\cdot e_{t(a)}$.  Since $Q$ is Schurian, this determines a unique arrow $\phi(a)$, with start $s(g \cdot a) = g \cdot e_{s(a)}$ and target $t(g \cdot a) = g \cdot e_{t(a)}$, such that $g\cdot a$ is a scalar times $\phi(a)$.  It is immediate from the definition of $\phi$ that $\phi$ is a quiver automorphism.
\end{proof}

\begin{convention}
We sometimes use $g\cdot a$ to label an arrow in a diagram, or refer to $g\cdot a$ as an arrow in exposition, in order to avoid introducing the extra notation $\phi$.  In these cases, it is understood that one must actually multiply by a scalar to get an arrow on the nose.
\end{convention}

Next, we study the action of skew-primitive elements on arrowless quivers $Q_0$, which then leads to Taft actions on the semisimple algebra $\kk Q_0$. Since the generator $x$ of $T(n)$ is $(1,g)$-skew-primitive,  the relation $e_i^2 = e_i$ of $\kk Q_0$ gives us that
\begin{equation} \label{eq:spanei}
x\cdot e_i= x \cdot (e_i^2) = e_i (x\cdot e_i) + (x\cdot e_i) (g \cdot e_i) \in {\rm span}_{\kk}\{e_i, g\cdot e_i\}.
\end{equation}
So, to study extensions of a $G$-action on $\kk Q_0$ to a $T(n)$-action, we can restrict ourselves to a single $G$-orbit of vertices.  From here on, we apply the above results to the case where $G=G(T(n))$ is the cyclic group generated by $g \in T(n)$, which we identify with $\mathbb{Z}_n$.

\begin{proposition} [$T(n)$-actions on $\kk Q_0$] \label{prop:TaftactQ0} 
Let $Q_0 = \{1, 2, \dotsc, m\}$ be the vertex set of a quiver, where $m$ divides $n$, and  $\mathbb{Z}_n$ acts on $\kk Q_0$ by $g\cdot e_i = e_{i+1}$. Here, subscripts are always interpreted modulo $m$.
\begin{enumerate}[(i)]
\item If $m < n$ (so the $\mathbb{Z}_n$-action on $Q_0$ is not faithful), then $x$ acts on $\kk Q_0$ by 0.
\item If $m=n$ (so $\mathbb{Z}_n$ acts faithfully on $Q_0$), then the action of $x$ on $\kk Q_0$ is exactly of the form
\begin{equation}\label{eq:xonvertex}
x \cdot e_i = \gamma \zeta^i (e_i - \zeta e_{i+1}) \text{\quad \quad for all $i$},
\end{equation}
where $\gamma \in \kk$ can be any scalar. 
\end{enumerate}
In particular, we can extend the action of $\mathbb{Z}_n$ on $\kk Q_0$ to an inner faithful action of $T(n)$ on $\kk Q_0$  if and only if $m = n$.
\end{proposition}

\begin{proof}
Assume that we have a $T(n)$-action on $\kk Q_0$ extended from the $\mathbb{Z}_n$-action on $Q_0$ in Lemma~\ref{lem:GactQ0}.  By~\eqref{eq:spanei}, we know that  $x\cdot e_i = \alpha_i e_i + \beta_i e_{i+1}$ for some scalars $\alpha_i, \beta_i \in \kk$.  Then, we have that
\begin{equation} \label{eq:x dot 1}
0 = x\cdot 1 = x \cdot \sum_{i=1}^m e_i = \sum_{i=1}^m \alpha_i e_i + \beta_i e_{i+1} = \sum_{i=1}^m (\alpha_i + \beta_{i-1} )e_i,
\end{equation} 
which gives $\beta_{i-1} = -\alpha_i$. (Here, $\sum_{i=1}^m \beta_i e_{i+1} = \sum_{i=1}^m \beta_{i-1} e_i$ by reindexing.) 
Now the relation $xg = \zeta gx$ applied to $e_i$ gives
\begin{equation} \label{eq:alpha i}
\alpha_{i+1} e_{i+1} - \alpha_{i+2}e_{i+2} = \zeta(\alpha_{i} e_{i+1} - \alpha_{i+1} e_{i+2}),
\end{equation}
so that $\alpha_{i+1} = \zeta\alpha_i$ for all $i$.  Setting $\gamma := \alpha_1\zeta^{-1}$ gives $\alpha_i = \zeta^i \gamma$, so that  \eqref{eq:xonvertex} holds whenever a $T(n)$-action exists. 
We have assumed that $m$ divides $n$, but on the other hand, 
$ x \cdot e_i = x \cdot e_{i+m}$ implies that $\gamma\zeta^i=\gamma \zeta^{i+m}$. Thus, $\gamma = \gamma \zeta^m$. Hence, whenever $m < n = ord(\zeta)$, we have $\gamma=0$, and $x$ acts by 0. This establishes (i).

On the other hand, suppose that $m=n$.  We will show that Equation \eqref{eq:xonvertex} defines a $T(n)$-action on $\kk Q_0$ for any $\gamma \in \kk$.  A simple substitution verifies that $x g \cdot e_i = \zeta g x \cdot e_i$. The fact that the $x$-action preserves the relations $e_i e_j = \delta_{ij} e_i$  and $\sum_{i=1}^{ n} e_i =1$ is also easy to check by substitution. The only tedious part is to show that $x^n$ acts on $\kk Q_0$ by 0, which is verified by Lemma~\ref{lem:x on ei} in the appendix of computations, using symmetric functions. Now, the $T(n)$-action  on $\kk Q_0$ is inner faithful when $\gamma$ is nonzero, by Lemma~\ref{lem:Taftfaithful}. Therefore, (ii) holds.\end{proof}

Now to study Taft actions on path algebras of Schurian quivers in general, we set the following notation.

\begin{notation}[$\sigma(a)$]  \label{not:sigma} 
Suppose we have a quiver $Q$ and an action of $\mathbb{Z}_n \simeq \langle g \rangle \subset T(n)$ on $\kk Q$.  Given an arrow $a \in Q_1$, we know that there exists at most one path of length less than or equal to 1 from $s(a)$ to $t(g\cdot a)$ since $Q$ is Schurian and loopless. Denote this path by $\sigma(a)$ if it exists, and set $\sigma(a) = 0$ otherwise.  

To be more explicit, consider the following case: when $g$ fixes neither $s(a)$ nor $t(a)$, and $g\cdot t(a) \neq s(a)$, then $\sigma(a)$ is either an arrow or 0, and can be visualized in the following diagram.
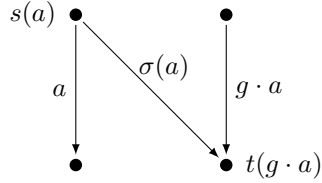
\[
\vcenter{\hbox{\begin{tikzpicture}[point/.style={shape=circle,fill=black,scale=.5pt,outer sep=3pt},>=latex]
   \node[point,label={left:$s(a)$}] (1) at (-1,2) {};
   \node[point] (2) at (1,2) {};
   \node[point] (3) at (-1,0) {};
   \node[point,label={right:$t(g\cdot a)$}] (4) at (1,0) {};
  \path[->]
  	(1) edge node[left] {$a$} (3) 
  	(2) edge node[right] {$g\cdot a$} (4)
  	(1) edge node[above] {\quad$\sigma(a)$} (4) ;
   \end{tikzpicture}}}
\]
If $g\cdot t(a) = s(a)$, then $\sigma(a)$ is the trivial path at $s(a)$.  Moreover, we have that $\sigma(a) = g\cdot a$ whenever $g\cdot s(a) = s(a)$, and $\sigma(a) = a$ whenever $g \cdot t(a) = t(a)$. 
\end{notation}

The following result determines the action of $x$ on any arrow of $Q$. 

\begin{proposition}\label{prop:Taftactarrow}
Suppose we have an action of $T(n)$ on $\kk Q$, and let $a \in Q_1$ with $i_+:=s(a)$ and $j_-:=t(a)$.
Then, there exist scalars $\alpha, \beta, \lambda \in \kk$ such that
\begin{equation}\label{eq:xactspan}
x \cdot a = \alpha a + \beta (g \cdot a) + \lambda \sigma(a).
\end{equation}
Moreover, $\alpha$, $\beta$, $\lambda$ can be determined in special cases depending on the relative configuration of $a$ and $g \cdot a$, as described in Figures~\ref{fig:xactona} and \ref{fig:xactona2} below. Here, the dotted red arrows indicate the action of $g$ on $Q_0$. 
\end{proposition}

\begin{figure}[h]
${\small
\begin{array}{|r|c|c|c|}
\hline
\begin{tabular}{c}\text{Relation}\\\text{between}\\{$a$ and $g\cdot a$}\end{tabular}
&
\vcenter{\hbox{\begin{tikzpicture}[point/.style={shape=circle,fill=black,scale=.5pt,outer sep=3pt},>=latex]
   \node[point,label={left:$i_+=g\cdot i_+$}] (1) at (0,2) {};
   \node[point,label={left:$j_-$}] (3) at (-1,0) {};
   \node[point,label={right:$g\cdot j_-$}] (4) at (1,0) {};
  \path[->]
  	(1) edge node[left] {$a$} (3) 
  	(1) edge node[right] {$g\cdot a = \sigma(a) $} (4);
  \draw[->][dotted,red] (3) edge (4);
 \draw[->][dotted,red] (1) to [out=140,in=60,looseness=10] (1);
   \end{tikzpicture}}}
&
\vcenter{\hbox{\begin{tikzpicture}[point/.style={shape=circle,fill=black,scale=.5pt,outer sep=3pt},>=latex]
   \node[point,label={left:$i_+$}] (1) at (-1,2) {};
   \node[point,label={right:$g\cdot i_+$}] (2) at (1,2) {};
   \node[point,label={left:$j_-=g\cdot j_-$}] (3) at (0,0) {};
  \path[->]
  	(1) edge node[left] {$a=\sigma(a)$} (3) 
  	(2) edge node[right] {$g\cdot a$} (3);
  \draw[->][dotted,red] (1) edge (2);
 \draw[->][dotted,red] (3) to [out=320,in=240,looseness=10] (3);
   \end{tikzpicture}}}
\quad
 &
\quad
\vcenter{\hbox{\begin{tikzpicture}[point/.style={shape=circle,fill=black,scale=.5pt,outer sep=3pt},>=latex]
   \node[point,label={left:$i_+$}] (1) at (-1,2) {};
   \node[point,label={right:$ g\cdot i_+$}] (2) at (1,2) {};
   \node[point,label={left:$j_-$}] (3) at (-1,0) {};
   \node[point,label={right:$g\cdot j_-$}] (4) at (1,0) {};
  \path[->]
  	(1) edge node[left] {$a$} (3) 
  	(2) edge node[right] {$g\cdot a$} (4)
  	(1) edge node[above] {\quad$\sigma(a)$} (4) ;
   \draw[->][dotted,red] (1) edge (2);
  \draw[->][dotted,red] (3) edge (4);
   \end{tikzpicture}}}
\quad
\\
\hline
&&&\\
x\cdot e_{i_+}= & 0 & (\zg_+)\zeta^i(e_{i_+} - \zeta g\cdot e_{i_+}) & (\zg_+)\zeta^i(e_{i_+} - \zeta g\cdot e_{i_+}) \\
\hline
&&&\\
x\cdot e_{j_-}= & (\zg_{-})\zeta^j(e_{j_-} - \zeta g\cdot e_{j_-}) & 0 & (\zg_{-})\zeta^j(e_{j_-} - \zeta g\cdot e_{j_-})\\
\hline
&&&\\
x\cdot a = & (\gamma_-)\zeta^j a + \beta (g\cdot a) & \alpha a -(\gamma_+) \zeta^{i+1} (g\cdot a) & (\gamma_-)\zeta^j a -(\gamma_+) \zeta^{i+1} (g\cdot a) +\lambda \sigma(a)\\
\hline
\end{array}}
$
\caption{$a(g \cdot a) = (g \cdot a) a= 0$}\label{fig:xactona}
\end{figure}

\begin{figure}[h]
${\small
\begin{array}{|r|c|c|c|}
\hline
\begin{tabular}{c}\text{Relation}\\\text{between}\\{$a$ and $g\cdot a$}\end{tabular}
&
\vcenter{\hbox{\begin{tikzpicture}[point/.style={shape=circle,fill=black,scale=.5pt,outer sep=3pt},>=latex]
   \node[point,label={right:$i$}] (1) at (-1,3) {};
   \node[point,label={right:$j = g\cdot i$}] (2) at (0,2) {};
   \node[point,label={right:$g\cdot j$}] (3) at (1,1) {};
  \path[->]
  	(1) edge node[right] {$a$} (2) 
  	(2) edge node[right] {$g\cdot a$} (3);
  \draw[->]	[bend right=50] (1) edge node[left] {$\sigma(a)$} (3);
  \draw[->][dotted,red,bend right=30] (1) edge (2);
  \draw[->][dotted,red,bend right=30] (2) edge (3);
   \end{tikzpicture}}}
&
\vcenter{\hbox{\hspace{.2in}\begin{tikzpicture}[point/.style={shape=circle,fill=black,scale=.5pt,outer sep=3pt},>=latex]
   \node[point,label={right:$g\cdot i$}] (1) at (-1,3) {};
   \node[point,label={right:$i = g\cdot j=\sigma(a)$}] (2) at (0,2) {};
   \node[point,label={right:$j$}] (3) at (1,1) {};
  \path[->]
  	(1) edge node[right] {$g\cdot a$} (2) 
  	(2) edge node[right] {$a$} (3);
  \draw[->][dotted,red,bend left=30] (2) edge (1);
  \draw[->][dotted,red,bend left=30] (3) edge (2);
   \end{tikzpicture}}}
\quad
\\
\hline
&&\\
x\cdot e_{i}= & \zg\zeta^i(e_{i} - \zeta g\cdot e_{i}) &  \zg\zeta^i(e_{i} - \zeta g\cdot e_{i}) \\
\hline
&&\\
x\cdot e_{j}= & \zg\zeta^j(e_{j} - \zeta g\cdot e_{j}) & \zg\zeta^j(e_{j} - \zeta g\cdot e_{j})  \\
\hline
&&\\
x\cdot a = & \gamma\zeta^j a -\gamma \zeta^{i+1} (g\cdot a) +\lambda \sigma(a)\ &  \gamma\zeta^j a -\gamma \zeta^{i+1} (g\cdot a) +\lambda e_i \\
\hline
\end{array}}
$
\caption{$a(g \cdot a) \neq 0~~$ or  $~~(g \cdot a)a \neq 0$}\label{fig:xactona2}
\end{figure}

\begin{proof}
Let $\zg_+, \zg_-\in \kk$ be the scalars from Proposition \ref{prop:TaftactQ0} such that 
\begin{equation} \label{eq:xonvertex2}
x \cdot e_{i_+} = (\zg_+)\zeta^i(e_{i_+} - \zeta g\cdot e_{i_+}) \quad \text{and} \quad x \cdot e_{j_-} = (\zg_{-})\zeta^j(e_{j_-} - \zeta g\cdot e_{j_-}),
\end{equation}
where $g \cdot e_{\ell_{\pm}} = e_{(\ell+1)_{\pm}}$ as usual.

From the relation $a=e_{i_+} a$, we can compute
\[
\begin{array}{ll}
x\cdot a ~=~ x \cdot (e_{i_+}a) &~=~ e_{i_+} (x\cdot a) + (\zg_+)\zeta^i(e_{i_+} - \zeta g\cdot e_{i_+})(g\cdot a)\\ &~=~ 
\begin{cases} 
e_{i_+} (x\cdot a) + (\gamma_+)\zeta^i(1 -  \zeta) (g\cdot a), & \text{if } e_{i_+}(g \cdot a) = g \cdot a\\
e_{i_+} (x\cdot a) - (\zg_+) \zeta^{i+1} (g\cdot a), & \text{if } e_{i_+}(g \cdot a) = 0.
\end{cases}
\end{array}
\]
Thus, $x \cdot a \in $ span$_{\kk}\{ \text{paths starting at } e_{i_+},~ g \cdot a\}$.
Similarly, the relation $a= ae_{j_-}$ gives
\[
\begin{array}{ll}
x\cdot a ~=~ x\cdot( a e_{j_-}) &~=~ (\zg_{-})\zeta^j a(e_{j_-} - \zeta g\cdot e_{j_-}) + (x\cdot a)(g\cdot e_{j_-})\\ &~=~
\begin{cases} 
(\zg_{-})\zeta^j(1 -\zeta)a + (x\cdot a)(g\cdot e_{j_-}), & \text{if } a (g \cdot e_{j_-}) =  a\\
(\zg_{-})\zeta^j a  + (x\cdot a)(g\cdot e_{j_-}), & \text{if } a (g \cdot e_{j_-})  = 0.
\end{cases}
\end{array}
\]
Thus, $x \cdot a \in $ span$_{\kk}\{a, ~\text{paths ending at } g \cdot e_{j_-}\}$. 
Intersecting these two conditions on $x \cdot a$ gives Equation~\eqref{eq:xactspan}.
The coefficients  $\alpha$, $\beta$, $\lambda$ can be determined more explicitly, but depend on the relative configuration of $a$ and $g\cdot a$.  We consider the various cases below.

Suppose that $e_{i_+}(g \cdot a) = g \cdot a$, then $g \cdot s(a) = s(a)$. So, as remarked in Notation \ref{not:sigma}, we have that $\sigma(a) = g\cdot a$. We can omit this term in \eqref{eq:xactspan} by absorbing $\lambda$ into $\beta$ in this case.
Similarly, if $a (g \cdot e_{j_-}) =  a$, then $g \cdot t(a) = t(a)$ and  we have that $\sigma(a) =a$. We can omit this term in \eqref{eq:xactspan} by absorbing $\lambda$ into $\alpha$ in this case. On the other hand, if $e_{i_+}(g \cdot a) = 0$, then $x \cdot a = x \cdot e_{i_+} a$ implies that
$$\alpha a + \beta (g \cdot a) + \lambda \sigma(a) ~=~ e_{i_+} [\alpha a + \beta (g \cdot a) + \lambda \sigma(a)] - (\gamma_{+}) \zeta^{i+1} (g\cdot a)
~=~ \alpha a  - (\gamma_+) \zeta^{i+1}(g \cdot a) + \lambda \sigma(a).$$
Thus in this case, $\beta=  - (\gamma_+) \zeta^{i+1}$.
Similarly, if $a (g \cdot e_{j_-}) =  0$, then $x \cdot a = x \cdot a e_{j_-}$ implies that
$$\alpha a + \beta (g \cdot a) + \lambda \sigma(a) ~=~ (\zg_{-})\zeta^ja + [\alpha a + \beta (g \cdot a) + \lambda \sigma(a)](g\cdot e_{j_-}) ~=~ (\zg_{-})\zeta^ja +  \beta (g \cdot a) + \lambda \sigma(a).$$
So, $\alpha = (\gamma_-)\zeta^j$.
These results are collected in Figure~\ref{fig:xactona}.  In each case, the $x$-action on the vertices follow from Proposition~\ref{prop:TaftactQ0}.  In Figure~\ref{fig:xactona2}, the results are further specialized to the cases where $g\cdot s(a) =t(a)$ or $g\cdot t(a) = s(a)$.  Since the $\pm$ notation serves to distinguish the orbits of $s(a)$ and $t(a)$, the $\pm$ notation is dropped in Figure~\ref{fig:xactona2}.
\end{proof}

As an illustration of the results above, we study Taft actions on a path algebra $\kk Q$, where $\mathbb{Z}_n$ fixes  $Q_0$.

 \begin{example}[$T(n)$-action on $\kk Q$, $\mathbb{Z}_n$ fixes $Q_0$]  \label{ex:GfixesQ0}
 If $\mathbb{Z}_n$ fixes the vertices of a quiver $Q$, then we claim that any extended action of $T(n)$ on $\kk Q$ is not inner faithful. First, by Proposition~\ref{prop:TaftactQ0} (i) with $m=1$, we get that $x \cdot e_i =0$ for all $i \in Q_0$. For the arrows, we get that   $s(g \cdot a) = e_{s(a)}$ and $t(g \cdot a) = e_{t(a)}$  by the assumption. So, $\sigma(a) = g \cdot a$. Moreover, $g \cdot a = \mu_a a$ by Lemma~\ref{lem:GactkQ}. Now, Proposition~\ref{prop:Taftactarrow} implies that $x \cdot a = \alpha a$ for some $\alpha \in \kk$. Finally, using the relation $x^n=0$, we conclude that $x \cdot a = 0$, and the claim holds.
\end{example}


\section{Minimal Quivers} \label{sec:minimalQ}

Given any Hopf algebra  action on an algebra $A$,  it is natural to study the restricted Hopf action on a subalgebra of $A$, if one exists. We introduce {\it $\mathbb{Z}_n$-minimal quivers} in this section, which will be the building blocks of Taft actions on path algebras of quivers in subsequent sections.
The following definition serves to fix notation for specific quivers that will be used throughout the rest of the paper.

\begin{definition}[$K_m$, $K_{m,m'}$, $a_j^i$, $b_j^i$] \label{def:complete}
Let $m$ and $m'$ be positive integers.  
\begin{enumerate}
\item[(1)] The \emph{complete quiver} $K_m$ (or \emph{complete digraph}) has vertex set $\{1, 2, \dotsc, m\}$, with an arrow $a^i_j$ from $i$ to $j$ for every ordered pair of distinct vertices $(i,j)$.  For uniformity in the formulas, we also take the symbol $a^i_i$ to mean the trivial path $e_i$ at vertex $i$ (rather than a loop, which we have excluded).

\item[(2)] The {\it complete bipartite quiver} $K_{m,m'}$ has a top row of $m$ vertices and a bottom row of $m'$ vertices, labeled by $\{1_+, \dots, m_+\}$ and $\{1_-, \dots, m'_-\}$, respectively.
There is an arrow $b^i_j$ from each vertex $i_+$ in the top row to each vertex $j_-$ in the bottom row; that is, $s(b^i_j) = i_+$ and $t(b^i_j) = j_-$. 
\end{enumerate}
\end{definition}

An example of each type is given below, without vertex or arrow labels.

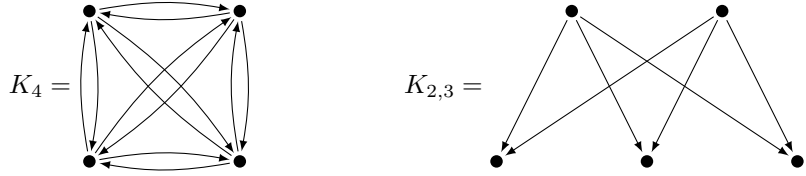
\[
\hspace{-.5in}
K_{4}=
\vcenter{\hbox{\begin{tikzpicture}[point/.style={shape=circle,fill=black,scale=.5pt,outer sep=3pt},>=latex]
   \node[point] (1) at (0,0) {};
   \node[point] (2) at (2,0) {};
   \node[point] (3) at (2,2) {};
   \node[point] (4) at (0,2) {};

\foreach \s in {1,...,4}
\foreach \t in {1,...,4}
{
	\ifthenelse{\NOT \s = \t}{\draw[->, bend left =10] (\s) edge (\t);} 
}
   \end{tikzpicture}}}
\hspace{.8in}
K_{2,3}=
\vcenter{\hbox{\begin{tikzpicture}[point/.style={shape=circle,fill=black,scale=.5pt,outer sep=3pt},>=latex]
   \node[point] (1) at (-1,2) {};
   \node[point] (2) at (1,2) {};
   \node[point] (3) at (-2,0) {};
   \node[point] (4) at (0,0) {};
   \node[point] (5) at (2,0) {};
  \path[->]
  	(1) edge  (3) 
  	(2) edge  (4)
  	(1) edge (4) 
  	(1) edge  (5)
	(2) edge  (5)
  	(2) edge (3) ;
   \end{tikzpicture}}}
\]
\medskip

To illustrate the arrow labels, the first diagram below shows some arrows in $K_4$ and the second diagram shows all arrows starting at $1^+$ in $K_{2,3}$.

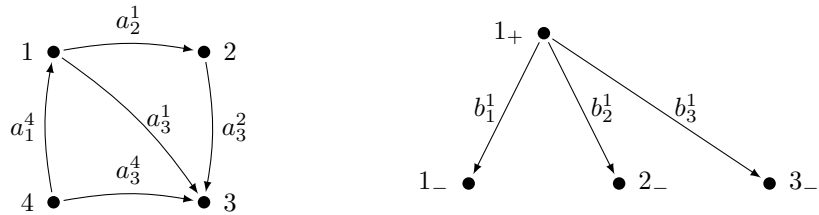
\[
\vcenter{\hbox{\begin{tikzpicture}[point/.style={shape=circle,fill=black,scale=.5pt,outer sep=3pt},>=latex]
   \node[point,label={left:1}] (1) at (0,2) {};
   \node[point,label={right:2}] (2) at (2,2) {};
   \node[point,label={right:3}] (3) at (2,0) {};
   \node[point,label={left: 4}] (4) at (0,0) {};
\path[->] 
(1) edge [bend left=10] node[above] {$a^1_2$}  (2) 
(2) edge [bend left=10] node[right] {$a^2_3$}(3)
(1) edge [bend left=10] node[right] {$a^1_3$}(3)
(4) edge [bend left=10] node[left] {$a^4_1$}(1)
 (4) edge [bend left=10] node[above] {$a^4_3$} (3);
   \end{tikzpicture}}}
\hspace{.8in}
\vcenter{\hbox{\begin{tikzpicture}[point/.style={shape=circle,fill=black,scale=.5pt,outer sep=3pt},>=latex]
   \node[point,label={left:$1_+$}] (0t) at (-1,2) {};
   \node[point,label={left:$1_-$}] (0b) at (-2,0) {};
   \node[point,label={right:$2_-$}] (1) at (0,0) {};
   \node[point,label={right:$3_-$}] (2) at (2,0) {};
  \path[->]
  	(0t) edge node[left] {$b^1_1$} (0b) 
 	(0t) edge node[right] {$b^1_2$} (1) 
  	(0t) edge node[right] {~$b^1_3$} (2); 
   \end{tikzpicture}}}
\]
\medskip

\noindent See Figure~\ref{fig:Z2minimal} for further illustrations.

Suppose we have an action of $\mathbb{Z}_n$ on the path algebra of a subquiver of $K_m$ or $K_{m, m'}$.  Let $g$ be a generator of $\mathbb{Z}_n$.  After possibly relabeling, we can assume $g$ acts on the trivial paths by $g\cdot e_i = e_{i+1}$ (for $K_m$) and $g\cdot e_{i_+} = e_{(i+1)_+},\ g \cdot e_{i_-= (i+1)_-}$ (for $K_{m, m'}$), where the indices are taken modulo $m$ or $m'$ as appropriate. 
 By Lemma \ref{lem:GactkQ}, there is a collection of nonzero scalars $\mu_{i,j}$ such that $g \cdot a^i_j =\mu_{i,j} a^{i+1}_{j+1}$ (for $K_m$), or $g \cdot b^i_j =\mu_{i,j} b^{i+1}_{j+1}$ (for $K_{m, m'}$).
Again, sub/super-scripts are interpreted modulo $m$ or $m'$ as appropriate.  Since $g^n$ is the identity and $g \cdot e_i = e_{i+1}$ (for $K_m$), we have that these scalars $\mu_{i,j}$ satisfy
\begin{equation}\label{eq:prodmu}
\begin{split}
\mu_{i,i} &=1 \quad  \text{in the }K_m \text{ case,}  \\
\textstyle \prod_{\ell=0}^{n-1} \mu_{i+\ell, j+\ell} &= 1 \quad \text{in both cases.}
\end{split}
\end{equation}

Since an arrow (or, more specifically, the source and target of an arrow) of a quiver $Q$ can only be part one or two $\mathbb{Z}_n$-orbits of $Q_0$, we make the following definition.

\begin{definition}[$\mathbb{Z}_n$-minimal, Type A, Type B] \label{def:typeAB}
Let $\mathbb{Z}_n$ act  on  a quiver $Q$. Say that $Q$ is $\mathbb{Z}_n${\it-minimal} of:
\begin{itemize}
\item {\it Type A} if $Q$ is a $\mathbb{Z}_n$-stable subquiver of $K_m$, where $m > 1$ is a positive integer dividing $n$; or
\item {\it Type B} if $Q$ is a $\mathbb{Z}_n$-stable subquiver of $K_{m,m'}$, where $m,m'\geq 1$ and $m,m'$ divide $n$. 
\end{itemize}
\end{definition}

Consider the following example.

\begin{example} \label{ex:Z4minl}
The following quiver $Q$ admits both a $\mathbb{Z}_4$-action and a $\mathbb{Z}_2$-action illustrated by the dotted red arrows below. Observe that $Q$ is $\mathbb{Z}_4$-minimal of Type B. Further, we see it is $\mathbb{Z}_{4d}$-minimal of Type B for any integer $d \geq 1$. However, it is not $\mathbb{Z}_2$-minimal of Type B. 

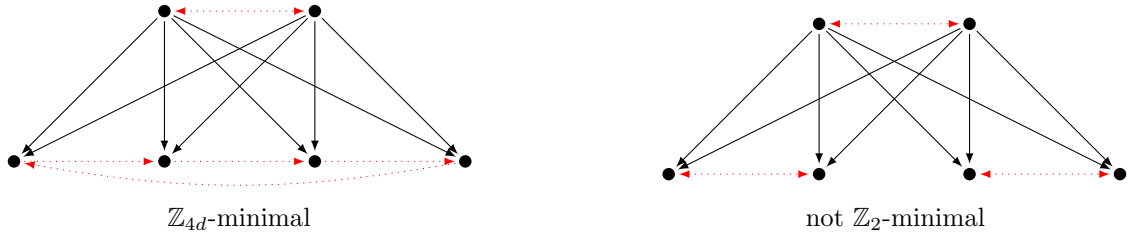
\[
 \vcenter{\hbox{\begin{tikzpicture}[point/.style={shape=circle,fill=black,scale=.5pt,outer sep=3pt},>=latex]
   \node[point] (1) at (-1,2) {};
   \node[point] (2) at (1,2) {};
   \node[point] (3) at (-3,0) {};
   \node[point] (4) at (-1,0) {};
 \node[point] (5) at (1,0) {};
   \node[point] (6) at (3,0) {};
\draw[->] (1) edge (3);
\draw[->] (1) edge (4);
\draw[->] (1) edge (5);
\draw[->] (1) edge (6);
\draw[->] (2) edge (3);
\draw[->] (2) edge (4);
\draw[->] (2) edge (5);
\draw[->] (2) edge (6);
\draw[<->][dotted,red] (1) edge (2);
\draw[->][dotted,red] (3) edge (4);
\draw[->][dotted,red] (4) edge (5);
\draw[->][dotted,red] (5) edge (6);
\draw[->][dotted,red,style={bend left=10}] (6) edge (3);
   \end{tikzpicture}}}
\hspace{1in}
 \vcenter{\hbox{\begin{tikzpicture}[point/.style={shape=circle,fill=black,scale=.5pt,outer sep=3pt},>=latex]
   \node[point] (1) at (-1,2) {};
   \node[point] (2) at (1,2) {};
   \node[point] (3) at (-3,0) {};
   \node[point] (4) at (-1,0) {};
 \node[point] (5) at (1,0) {};
   \node[point] (6) at (3,0) {};
\draw[->] (1) edge (3);
\draw[->] (1) edge (4);
\draw[->] (1) edge (5);
\draw[->] (1) edge (6);
\draw[->] (2) edge (3);
\draw[->] (2) edge (4);
\draw[->] (2) edge (5);
\draw[->] (2) edge (6);
\draw[<->][dotted,red] (1) edge (2);
\draw[<->][dotted,red] (3) edge (4);
\draw[<->][dotted,red] (5) edge (6);
   \end{tikzpicture}}}
\]
\begin{center} \hspace{.05in} $\mathbb{Z}_{4d}$-minimal \hspace{2.5in} not $\mathbb{Z}_2$-minimal \end{center}
\end{example}
\medskip

Now we list the $\mathbb{Z}_n$-minimal quivers for small $n$. For Type A, note that there is only one $\mathbb{Z}_2$-minimal quiver that admits a transitive action of $\mathbb{Z}_2$ on vertices; see {\sf (I)} of  Figure~\ref{fig:Z2minimal}. See Figure \ref{fig:Z3minimalA} for the three $\mathbb{Z}_3$-minimal quivers of  Type A.  The dotted, red arrow indicates the action of $\mathbb{Z}_3$ on $Q_0$ (clockwise rotation in each case).
\medskip

\begin{figure}[h]
\[
 \vcenter{\hbox{\begin{tikzpicture}[point/.style={shape=circle,fill=black,scale=.5pt,outer sep=3pt},>=latex]
   \node[point] (1) at (0,1) {};
   \node[point] (2) at (-1,-1) {};
   \node[point] (3) at (1,-1) {};
\draw[->][bend left=20]  (1) edge (3);
\draw[->][bend left=20]  (3) edge (2);
\draw[->][bend left=20]  (2) edge (1);
\path[->,red,dotted]
 (1) edge (3)
 (3) edge (2)
 (2) edge (1); 
   \end{tikzpicture}}}
\hspace{1in}
 \vcenter{\hbox{\begin{tikzpicture}[point/.style={shape=circle,fill=black,scale=.5pt,outer sep=3pt},>=latex]
   \node[point] (1) at (0,1) {};
   \node[point] (2) at (-1,-1) {};
   \node[point] (3) at (1,-1) {};
\draw[->][bend left=20]  (1) edge (2);
\draw[->][bend left=20]  (2) edge (3);
\draw[->][bend left=20]  (3) edge (1);
\path[->,red,dotted]
 (1) edge (3)
 (3) edge (2)
 (2) edge (1); 
   \end{tikzpicture}}}
\hspace{1in}
 \vcenter{\hbox{\begin{tikzpicture}[point/.style={shape=circle,fill=black,scale=.5pt,outer sep=3pt},>=latex]
   \node[point] (1) at (0,1) {};
   \node[point] (2) at (-1,-1) {};
   \node[point] (3) at (1,-1) {};
\draw[->][bend left=20]  (1) edge (3);
\draw[->][bend left=20]  (3) edge (2);
\draw[->][bend left=20]  (2) edge (1);
\draw[->][bend left=20]  (1) edge (2);
\draw[->][bend left=20]  (2) edge (3);
\draw[->][bend left=20]  (3) edge (1);
\path[->,red,dotted]
 (1) edge (3)
 (3) edge (2)
 (2) edge (1); 
   \end{tikzpicture}}}
\]
\caption{$\mathbb{Z}_3$-minimal quivers of  Type A}\label{fig:Z3minimalA}
\end{figure}
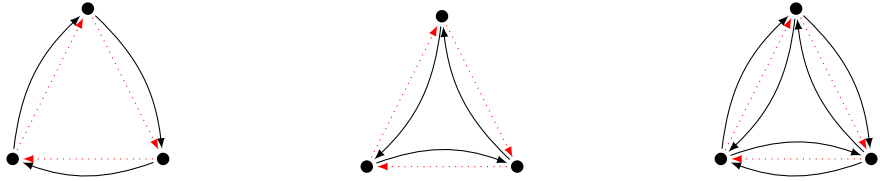

\noindent For Type B,  there  are five Type B $\mathbb{Z}_2$-minimal quivers; see Figure~\ref{fig:Z2minimal} {\sf (II)}-{\sf (VI)} for an illustration.  Moreover, see  Figure \ref{fig:Z3minimalB} for the six $\mathbb{Z}_3$-minimal quivers of Type B.

\begin{figure}[h]
\[
\hspace{.6in}
 \vcenter{\hbox{\begin{tikzpicture}[point/.style={shape=circle,fill=black,scale=.5pt,outer sep=3pt},>=latex]
   \node[point] (1) at (0,1) {};
   \node[point] (2) at (0,-1) {};
\draw[->]  (1) edge (2);
 \draw[->][dotted,red] (1) to [out=140,in=60,looseness=10] (1);
 \draw[->][dotted,red] (2) to [out=320,in=240,looseness=10] (2);
   \end{tikzpicture}}}
\hspace{1.1in}
 \vcenter{\hbox{\begin{tikzpicture}[point/.style={shape=circle,fill=black,scale=.5pt,outer sep=3pt},>=latex]
   \node[point] (1) at (0,1) {};
   \node[point] (2) at (-2,-1) {};
   \node[point] (3) at (0,-1) {};
   \node[point] (4) at (2,-1) {};
\draw[->]  (1) edge (2);
\draw[->]  (1) edge (3);
\draw[->]  (1) edge (4);
\draw[->][dotted,red] (1) to [out=140,in=60,looseness=10] (1);
\draw[->][dotted,red] (2) to  (3);
\draw[->][dotted,red] (3) to  (4);
\draw[->][dotted,red,bend left=20] (4) to  (2); 
   \end{tikzpicture}}}
\hspace{.5in}
 \vcenter{\hbox{\begin{tikzpicture}[point/.style={shape=circle,fill=black,scale=.5pt,outer sep=3pt},>=latex]
   \node[point] (1) at (0,-1) {};
   \node[point] (2) at (-2,1) {};
   \node[point] (3) at (0,1) {};
   \node[point] (4) at (2,1) {};
\draw[->]  (2) edge (1);
\draw[->]  (3) edge (1);
\draw[->]  (4) edge (1);
\draw[->][dotted,red] (1) to [out=320,in=240,looseness=10] (1);
\draw[->][dotted,red] (2) to  (3);
\draw[->][dotted,red] (3) to  (4);
\draw[->][dotted,red,bend right=20] (4) to  (2); 
   \end{tikzpicture}}}
\]
\[
 \vcenter{\hbox{\begin{tikzpicture}[point/.style={shape=circle,fill=black,scale=.5pt,outer sep=3pt},>=latex]
   \node[point] (1) at (-2,1) {};
   \node[point] (2) at (0,1) {};
   \node[point] (3) at (2,1) {};
   \node[point] (4) at (-2,-1) {};
   \node[point] (5) at (0,-1) {};
   \node[point] (6) at (2,-1) {};
\draw[->] (1) edge (4);
\draw[->] (2) edge (5);
\draw[->] (3) edge (6);
\draw[->][dotted,red] (1) to  (2);
\draw[->][dotted,red] (2) to  (3);
\draw[->][dotted,red,bend right=20] (3) to  (1); 
\draw[->][dotted,red] (4) to  (5);
\draw[->][dotted,red] (5) to  (6);
\draw[->][dotted,red,bend left=20] (6) to  (4); 
   \end{tikzpicture}}}
\hspace{.5in}
\vcenter{\hbox{\begin{tikzpicture}[point/.style={shape=circle,fill=black,scale=.5pt,outer sep=3pt},>=latex]
   \node[point] (1) at (-2,1) {};
   \node[point] (2) at (0,1) {};
   \node[point] (3) at (2,1) {};
   \node[point] (4) at (-2,-1) {};
   \node[point] (5) at (0,-1) {};
   \node[point] (6) at (2,-1) {};
\draw[->] (1) edge (4);
\draw[->] (2) edge (5);
\draw[->] (3) edge (6);
\draw[->] (1) edge (5);
\draw[->] (2) edge (6);
\draw[->] (3) edge (4);
\draw[->][dotted,red] (1) to  (2);
\draw[->][dotted,red] (2) to  (3);
\draw[->][dotted,red,bend right=20] (3) to  (1); 
\draw[->][dotted,red] (4) to  (5);
\draw[->][dotted,red] (5) to  (6);
\draw[->][dotted,red,bend left=20] (6) to  (4); 
   \end{tikzpicture}}}
\hspace{.5in}
\vcenter{\hbox{\begin{tikzpicture}[point/.style={shape=circle,fill=black,scale=.5pt,outer sep=3pt},>=latex]
   \node[point] (1) at (-2,1) {};
   \node[point] (2) at (0,1) {};
   \node[point] (3) at (2,1) {};
   \node[point] (4) at (-2,-1) {};
   \node[point] (5) at (0,-1) {};
   \node[point] (6) at (2,-1) {};
\draw[->] (1) edge (4);
\draw[->] (2) edge (5);
\draw[->] (3) edge (6);
\draw[->] (1) edge (5);
\draw[->] (2) edge (6);
\draw[->] (3) edge (4);
\draw[->] (1) edge (6);
\draw[->] (2) edge (4);
\draw[->] (3) edge (5);
\draw[->][dotted,red] (1) to  (2);
\draw[->][dotted,red] (2) to  (3);
\draw[->][dotted,red,bend right=20] (3) to  (1); 
\draw[->][dotted,red] (4) to  (5);
\draw[->][dotted,red] (5) to  (6);
\draw[->][dotted,red,bend left=20] (6) to  (4); 
   \end{tikzpicture}}}
\]
\caption{$\mathbb{Z}_3$-minimal quivers of  Type B}\label{fig:Z3minimalB}
\end{figure}
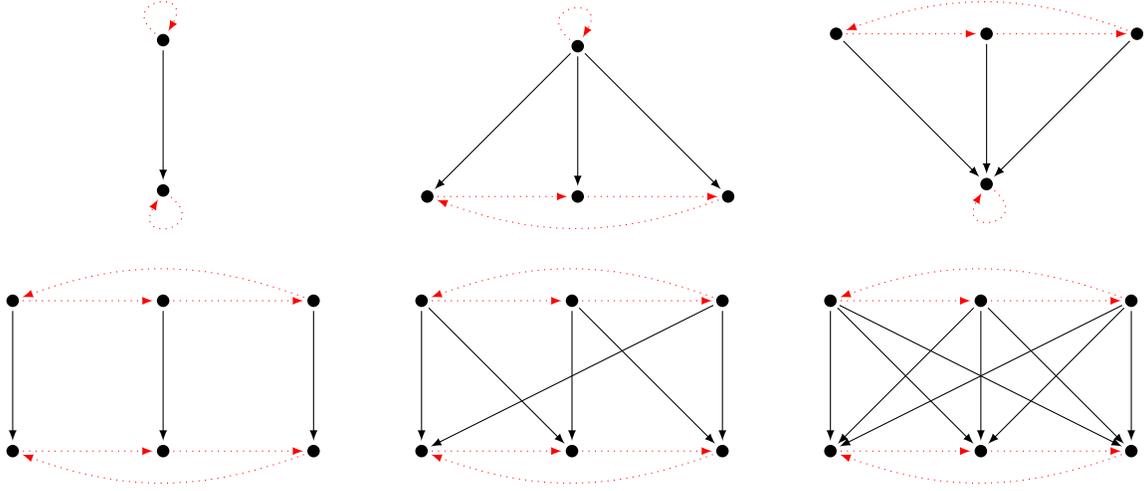


\section{Sweedler algebra actions on path algebras of minimal quivers} \label{sec:Sweedler}

In this section, we study the action of the Sweedler algebra $T(2)$ on path algebras of quivers. This is achieved by first computing the action of the Sweedler algebra on $\mathbb{Z}_2$-minimal quivers [Theorem~\ref{thm:T(2)minimal}]. Then, later, we present results on gluing such actions to yield Sweedler actions on more general quivers in Section~\ref{sec:gluing}.

Recall from Definition~\ref{def:Taft} that the Sweedler algebra is the 4-dimensional Taft algebra $T(2)$ generated by a grouplike element $g$ and a $(1,g)$-skew-primitive element $x$, subject to relations: $$g^2 =1, \quad x^2 = 0, \quad xg + gx=0.$$
Note that $G(T(2)) \simeq \mathbb{Z}_2$.

\begin{figure}[h]
{\footnotesize \[
\begin{array}{|c|c|c|}
\hline
\quad
\vcenter{\hbox{\begin{tikzpicture}[point/.style={shape=circle,fill=black,scale=.5pt,outer sep=3pt},>=latex]
   \node[point,label={left:$1$}] (1) at (-1,2) {};
   \node[point,label={right:$2$}] (2) at (1,2) {};
  \path[->]
  	(1) edge[bend left=25] node[above] {$a_2^1$} (2) 
  	(2) edge[bend left=25] node[below] {$a_1^2$} (1);
  \draw[<->][dotted,red] (1) edge (2);
   \end{tikzpicture}}}
\quad
&
\quad
\hspace{-.07in}\vcenter{\hbox{\begin{tikzpicture}[point/.style={shape=circle,fill=black,scale=.5pt,outer sep=3pt},>=latex]
   \node[point,label={left:$1_+$}] (1) at (-1,2) {};
   \node[point,label={left:$1_-$}] (2) at (-1,0) {};
  \path[->]
  	(1) edge node[left] {$b_1^1$} (2) ;
 \draw[->][dotted,red] (1) to [out=140,in=60,looseness=10] (1);
 \draw[->][dotted,red] (2) to [out=320,in=240,looseness=10] (2);
   \end{tikzpicture}}}
\quad
 &
\quad
\vcenter{\hbox{\begin{tikzpicture}[point/.style={shape=circle,fill=black,scale=.5pt,outer sep=3pt},>=latex]
   \node[point,label={left:$1_+=2_+$}] (1) at (0,2) {};
   \node[point,label={left:$1_-$}] (3) at (-1,0) {};
   \node[point,label={right:$2_-$}] (4) at (1,0) {};
  \path[->]
  	(1) edge node[left] {$b_1^1$} (3) 
  	(1) edge node[right] {$b_2^2$} (4);
  \draw[<->][dotted,red] (3) edge (4);
 \draw[->][dotted,red] (1) to [out=140,in=60,looseness=10] (1);
   \end{tikzpicture}}}
\quad\\
{\sf (I)} & {\sf (II)} & {\sf (III)}\\
&&\\
\hline
&&\\
\quad
\vcenter{\hbox{\begin{tikzpicture}[point/.style={shape=circle,fill=black,scale=.5pt,outer sep=3pt},>=latex]
   \node[point,label={left:$1_+$}] (1) at (-1,2) {};
   \node[point,label={right:$2_+$}] (2) at (1,2) {};
   \node[point,label={left:$1_-=2_-$}] (3) at (0,0) {};
  \path[->]
  	(1) edge node[left] {$b_1^1$} (3) 
  	(2) edge node[right] {$b_2^2$} (3);
  \draw[<->][dotted,red] (1) edge (2);
 \draw[->][dotted,red] (3) to [out=320,in=240,looseness=10] (3);
   \end{tikzpicture}}}
\quad
 &
\quad
\vcenter{\hbox{\begin{tikzpicture}[point/.style={shape=circle,fill=black,scale=.5pt,outer sep=3pt},>=latex]
   \node[point,label={left:$1_+$}] (1) at (-1,2) {};
   \node[point,label={right:$2_+$}] (2) at (1,2) {};
   \node[point,label={left:$1_-$}] (3) at (-1,0) {};
   \node[point,label={right:$2_-$}] (4) at (1,0) {};
  \path[->]
  	(1) edge node[left] {$b_1^1$} (3) 
  	(2) edge node[right] {$b_2^2$} (4);
   \draw[<->][dotted,red] (1) edge (2);
  \draw[<->][dotted,red] (3) edge (4);
   \end{tikzpicture}}}
\quad
& 
\quad \vcenter{\hbox{\begin{tikzpicture}[point/.style={shape=circle,fill=black,scale=.5pt,outer sep=3pt},>=latex]
   \node[point,label={left:$1_+$}] (1) at (-1,2) {};
   \node[point,label={right:$2_+$}] (2) at (1,2) {};
   \node[point,label={left:$1_-$}] (3) at (-1,0) {};
   \node[point,label={right:$2_-$}] (4) at (1,0) {};
  \path[->]
  	(1) edge node[left] {$b_1^1$} (3) 
  	(2) edge node[right] {$b_2^2$} (4)
  	(1) edge node[pos=0.25,above] {$b_2^1$} (4) 
  	(2) edge node[pos=0.25,above] {$b_1^2$} (3) ;
  \draw[<->][dotted,red] (1) edge (2);
  \draw[<->][dotted,red] (3) edge (4);
   \end{tikzpicture}}}
\quad \\
{\sf (IV)} & {\sf (V)} & {\sf (VI)}\\
&&\\
\hline
\end{array}
\]}
\caption{The $\mathbb{Z}_2$-minimal quivers}\label{fig:Z2minimal}
\end{figure}
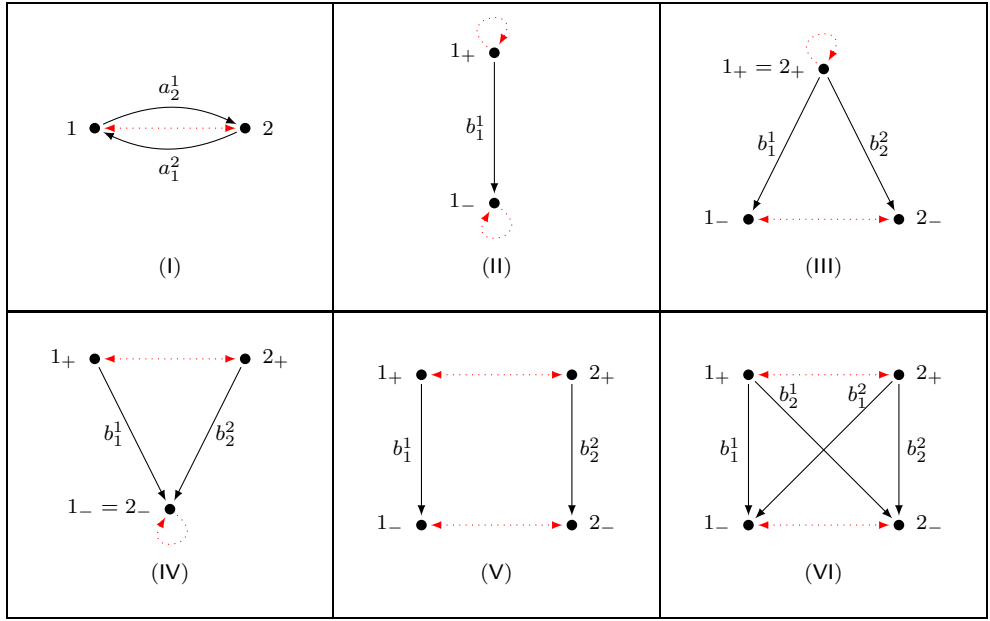

\medskip

\newpage

\begin{theorem} \label{thm:T(2)minimal} For each quiver $Q$ in Figure~\ref{fig:Z2minimal} (where the dotted red arrow illustrates the $g$-action on vertices), we have that the $\mathbb{Z}_2$-action on $Q$ extends to an action of the Sweedler algebra $T(2)$ on $\kk Q$ precisely as follows. We take $\gamma$ (resp., $\gamma_+$ and $\gamma_-$) to be the scalar from the $x$-action on $e_i$ (resp., on $e_{i_+}$ and on $e_{j_-}$) in  \eqref{eq:xonvertex} (resp., in \eqref{eq:xonvertex2}).

\begin{center}
{\small {\begin{tabular}{|c|lllll|}
\hline
$Q$ & $x$-action on $Q_0$ && $x$-action on $Q_1$ && \\
\hline
{\sf (I)}  
&\begin{tabular}{l}
$x \cdot e_{1} = -\gamma(e_{1}+e_{2})$\\
 $x \cdot e_{2} = \gamma(e_{1}+e_{2})$
\end{tabular}
&& 
\hspace{-.12in} \begin{tabular}{l}
$x \cdot a_2^1 = \gamma a^1_2-\gamma \mu a^2_1+\lambda e_{1}$\\
$x\cdot a_1^2 = \gamma \mu^{-1} a^1_2-\gamma a^2_1 -\lambda \mu^{-1} e_{2}$
\end{tabular}
 &&for $\lambda \in \kk$, $\mu \in \kk^{\times}$\\
\hline
{\sf (II)} 
&$x \cdot e_{1_+} = x \cdot e_{1_-} = 0$ 
&&$x \cdot b_1^1 = 0$&&\\
\hline
{\sf (III)} 
& \hspace{-.1in} \begin{tabular}{l}
$x \cdot e_{1_+} =0$\\ $x \cdot e_{1_-} = -(\gamma_-)(e_{1_-}+e_{2_-})$\\
 $x \cdot e_{2_-} = (\gamma_-)(e_{1_-}+e_{2_-})$
\end{tabular}
&& 
\hspace{-.07in}\begin{tabular}{l}
$x \cdot b_1^1 = -(\gamma_-) b_1^1 +\beta \mu b_2^2$\\
$x \cdot b_2^2 = -\beta \mu^{-1} b_1^1 + (\gamma_-) b_2^2$
\end{tabular}
&& for  $\beta^2 = (\gamma_-)^2$, $\mu \in \kk^{\times}$
\\
\hline
{\sf (IV)} 
& \hspace{-.1in} \begin{tabular}{l}
$x \cdot e_{1_+} = -(\gamma_+)(e_{1_+}+e_{2_+})$\\
$x \cdot e_{2_+} = (\gamma_+)(e_{1_+}+e_{2_+})$\\
$x \cdot e_{1_-} =0$
\end{tabular}
&& 
\hspace{-.1in} \begin{tabular}{l}
$x \cdot b_1^1 = \alpha b_1^1- (\gamma_{+}) \mu b_2^2$\\
$x \cdot b_2^2 = (\gamma_{+}) \mu^{-1} b_1^1 -\alpha b_2^2$
\end{tabular}
&& for $\alpha^2 = (\gamma_+)^2$, $\mu \in \kk^{\times}$ \\
\hline
{\sf (V)} 
&  \hspace{-.07in}\begin{tabular}{l}
$x \cdot e_{1_+} =  -(\gamma_+)(e_{1_+}+e_{2_+})$ \\
$x \cdot e_{2_+} = (\gamma_+)(e_{1_+}+e_{2_+})$\\
$x \cdot e_{1_-} = - (\gamma_-)(e_{1_-}+e_{2_-})$\\
$x \cdot e_{2_-} = (\gamma_-)(e_{1_-}+e_{2_-})$
\end{tabular}
&& 
\hspace{-.05in}\begin{tabular}{l}
$x \cdot b_1^1 = -(\gamma_-) b_1^1 - (\gamma_{+}) \mu b_2^2$\\
$x \cdot b_2^2 = (\gamma_{+}) \mu^{-1} b_1^1 + (\gamma_-) b_2^2$
\end{tabular}
&& for $(\gamma_+)^2 = (\gamma_{-})^2$, $\mu \in \kk^{\times}$\\
\hline
{\sf (VI)} 
&  \hspace{-.05in}\begin{tabular}{l}
$x \cdot e_{1_+} = -(\gamma_+)(e_{1_+}+e_{2_+})$ \\
$x \cdot e_{2_+} = (\gamma_+)(e_{1_+}+e_{2_+})$ \\
$x \cdot e_{1_-} = - (\gamma_-)(e_{1_-}+e_{2_-})$\\
$x \cdot e_{2_-} = (\gamma_-)(e_{1_-}+e_{2_-})$
\end{tabular}
&& 
\hspace{-.1in} \begin{tabular}{l}
$x \cdot b_1^1 =-(\gamma_-) b_1^1- (\gamma_{+}) \mu b_2^2+ \lambda b_2^1$\\
$x \cdot b_2^2 = (\gamma_{+}) \mu^{-1} b_1^1 + (\gamma_-)b^2_2 -\lambda \mu^{-1}\mu' b_1^2$\\
$x\cdot b^1_2 = (\gamma_-)b^1_2 - (\gamma_+) \mu' b^2_1 + \lambda' b^1_1$\\
$x\cdot b^2_1 = (\gamma_+)\mu'^{-1} b^1_2 - (\gamma_-)b^2_1 - \lambda' \mu \mu'^{-1} b^2_2$\\
\end{tabular}
&& 
\hspace{-.13in} \begin{tabular}{l}
for $(\gamma_+)^2 = (\gamma_{-})^2+\lambda \lambda'$\\
with $\lambda, \lambda' \in \kk$\\
$\mu, \mu' \in \kk^{\times}$ 
\end{tabular}\\
\hline
\end{tabular}}}
\end{center}
\end{theorem}
\medskip

\begin{proof}
The $x$-action  on vertices follows from \eqref{eq:xonvertex} and \eqref{eq:xonvertex2}.

For {\sf (I)}, we are in the situation illustrated in the second column of Figure \ref{fig:xactona2}, where additionally $g\cdot i = j$. So, we have that 
$$
x \cdot a^1_2 = \gamma a^1_2- \gamma {\mu_{1,2}} a^2_1 + \lambda e_{1}
$$
from  Proposition~\ref{prop:Taftactarrow} and Lemma~\ref{lem:GactkQ}.
Then using the relation $xg = -gx$ in $T(2)$, we find that 
\[
{\mu_{1,2}}x\cdot a^2_1= x \cdot (g \cdot a^1_2) = - g \cdot (x \cdot a^1_2) = -\gamma \mu_{1,2} a^2_1+\gamma  \mu_{1,2} \mu_{2,1} a^1_2-\lambda  \mu_{1,1} e_2.
\]
Using that $\mu_{1,2}^{-1} = \mu_{2,1}$ and $\mu_{1,1} = 1$ from Equation~\ref{eq:prodmu}, we obtain:
$$x \cdot a^2_1 = \gamma \mu_{2,1} a^1_2 - \gamma a^2_1 - \lambda \mu_{2,1} e_2.$$
The $x$-action on all other relations in $\kk Q$ yields tautologies, and the relations $x^2 = xg+gx=0$ are satisfied when applied to the arrows $a^1_2$ and $a^2_1$. Taking $\mu = \mu_{1,2}$ and $\mu^{-1} = \mu_{2,1}$, we see the action is of the claimed form.
\smallskip

For {\sf (II)}, we apply Example~\ref{ex:GfixesQ0}.
\smallskip

For {\sf (III)}, examining Figure \ref{fig:xactona}, we see that $x\cdot b_1^1 = -(\gamma_-) b_1^1 + \beta \mu_{1,1} b_2^2$ for some $\beta \in \kk$.
The relation $xg = -gx$ in $T(2)$ then gives $ \mu_{1,1} x\cdot b_2^2 = x \cdot (g\cdot b_1^1) = -g \cdot (x\cdot b_1^1) = -\beta \mu_{1,1}\mu_{2,2} b_1^1 + (\gamma_-) \mu_{1,1} b_2^2$.
Now applying the relation $x^2 =0$ of $T(2)$ to the arrow $b_1^1$ implies that $(\gamma_-)^2 = \beta^2$, and we take $\mu = \mu_{1,1}$.
\smallskip

For {\sf (IV)}, we get from Figure \ref{fig:xactona} that $x\cdot b_1^1 = \alpha b_1^1 - (\gamma_+)\mu_{1,1}b_2^2$ for some $\alpha \in \kk$.
The relation $xg = -gx$ in $T(2)$ then gives $\mu_{1,1} x\cdot b_2^2 = x \cdot (g\cdot b_1^1) = -g \cdot (x\cdot b_1^1) = -\alpha \mu_{1,1}b_2^2 + (\gamma_+)\mu_{1,1}\mu_{2,2} b_1^1$.
Now applying the relation $x^2 =0$  to the arrow $b_1^1$ implies that $(\gamma_+)^2 = \alpha^2$, and again we take $\mu = \mu_{1,1}$.
\smallskip

For {\sf (V)}, examining Figure \ref{fig:xactona} we find that $x \cdot b_1^1 = -(\gamma_-)b_1^1 - (\gamma_+)\mu_{1,1}b_2^2$.
Then the relation $xg = -gx$ in $T(2)$ gives that $\mu_{1,1}x \cdot b_2^2 = (\gamma_{+})\mu_{1,1}\mu_{2,2} b_1^1 + (\gamma_-) \mu_{1,1}b_2^2$. 
The relation $x^2=0$ applied to the arrows $b_1^1$ and $b_2^2$ implies that $(\gamma_+)^2 = (\gamma_-)^2$, and again we take $\mu = \mu_{1,1}$.
\smallskip

For {\sf (VI)}, we can see from Figure \ref{fig:xactona} that $x \cdot b_1^1 =  -(\gamma_-) b_1^1 - (\gamma_{+})\mu_{1,1}b_2^2 +\lambda b_2^1$ for some $\lambda \in \kk$. Similarly, we can see from this figure that
$x\cdot b^1_2 = (\gamma_-)b^1_2 - (\gamma_+)\mu_{1,2}b^2_1 + \lambda' b^1_1$ for some $\lambda' \in \kk$.  The relation $xg = -gx$ gives the formulas for $x\cdot b^2_2$ and $x\cdot b^2_1$ as claimed.
 Finally, the relation $x^2=0$ implies that $(\gamma_+)^2 = (\gamma_-)^2 + \lambda\lambda'$. We have taken $\mu =\mu_{1,1}$ and $\mu' = \mu_{1,2}$.
\end{proof}

\begin{remark} \label{rem:IIItoIV} Since $Q({\sf III})^{op}$ = $Q({\sf IV})$, we can get the Sweedler action on the path algebra of $Q({\sf IV})$ from its action on the path algebra of $Q({\sf III})$ using \eqref{eq:oppact}; see Remark~\ref{rem:Taftopp}. For instance, take $b^1_1 \in Q({\sf IV})$:
$$x \diamond b^1_1 ~=~ g^{-1}x \cdot b^1_1 ~=~ g^{-1} \cdot (-(\gamma_-)b^1_1+\beta \mu b^2_2) ~=~ -(\gamma_-) \mu b^2_2 + \beta b^1_1.$$
Now by identifying $\gamma_-$ with $\gamma_+$, and $\beta$ with $\alpha$, we get the desired action.
\end{remark}

\section{Taft algebra actions on path algebras of minimal quivers} \label{sec:Taft}
 
In this section, we give a complete description of $T(n)$-actions on path algebras of $\mathbb{Z}_n$-minimal quivers. 
Since the action of the grouplike elements of $T(n)$ has already been determined in Lemma~\ref{lem:GactkQ}, and the $T(n)$-action on the vertices has already been determined in Proposition \ref{prop:TaftactQ0}, all that remains is to describe possible actions of $x$ on the arrows.

\subsection{Taft actions on Type A $\mathbb{Z}_n$-minimal quivers} 
Let $Q$ be a $\mathbb{Z}_n$-minimal quiver of Type A, so that $Q$ is a subquiver of $K_m$ for some $m|n$ by Definition \ref{def:typeAB}.
Recall that the path algebra $\kk K_m$ has basis $a^i_j$ for $1 \leq i, j \leq m$, with $a^i_i=e_i$ being the trivial path at vertex $i$.

\begin{theorem} \label{thm:T(n)typeA} Retain the notation above.
Any $T(n)$-action on the path algebra of  Type $A$ $\mathbb{Z}_n$-minimal quiver $Q$ is given by
\begin{equation} \label{eq:A}
x\cdot a^i_j ~=~ \gamma(\zeta^j a^i_j - \zeta^{i+1} \mu_{i, j} a^{i+1}_{j+1}) + \lambda_{i, j} a^i_{j+1}.
\end{equation}
where $\gamma \in \kk$ from \eqref{eq:xonvertex}, $\mu_{i,j} \in \kk^{\times}$ from \eqref{eq:prodmu}, and $\lambda_{i, j} \in \kk$ are all scalars satisfying 
\begin{itemize}
\item $\mu_{i,i} = 1$ for all $i$, and $\prod_{\ell=0}^{n-1} \mu_{i+\ell, j+\ell} = 1$;
\item $\lambda_{i, j} = 0$ if either $i=j$ or the arrow $a^i_{j+1}$ does not exist in $Q$; and
\item $\lambda_{i+1,j+1} \mu_{i,j} = \zeta \lambda_{i,j}\mu_{i,j+1}$.
\end{itemize}
The superindices and subindices of arrows and scalars are taken modulo $m$.  
\end{theorem}

\begin{proof}
We have already used the relations of the path algebra to find that
$$x \cdot a^i_j ~=~ \gamma (\zeta^j a^i_j - \zeta^{i+1} \mu_{i,j} a^{i+1}_{j+1}) + \lambda_{i,j} a^i_{j+1},$$
for some scalars $\mu_{i,j} \in \kk^{\times}$ and $\lambda_{i,j}  \in \kk$ by Lemma~\ref{lem:GactkQ}, Propositions~\ref{prop:TaftactQ0} and~\ref{prop:Taftactarrow}. Namely,  the case of $i=j$ is Proposition~\ref{prop:TaftactQ0}, which gives $\lambda_{i, i}=0$ and $\mu_{i,i}=1$ for all $i$. Moreover, the case of $i \neq j$ is Proposition~\ref{prop:Taftactarrow}, which gives that $\lambda_{i,j} = 0$ if the arrow $a^i_{j+1}$ does not exist in $Q$.  The condition on the $\{\mu_{i,j}\}$ is from \eqref{eq:prodmu}.

The relation $xg = \zeta gx$ of $T(n)$ applied to $a^i_j$ gives on the one hand that
\[
xg\cdot a^i_j ~=~ x \cdot \mu_{i,j} a^{i+1}_{j+1} ~=~ \gamma \mu_{i,j} (\zeta^{j+1} a^{i+1}_{j+1} - \zeta^{i+2}  \mu_{i+1,j+1} a^{i+2}_{j+2}) + \lambda_{i+1, j+1} \mu_{i,j} a^{i+1}_{j+2},
\]
while on the other hand that
\[
\zeta gx\cdot a^i_j ~=~ \zeta\gamma(\zeta^j \mu_{i,j} a^{i+1}_{j+1} - \zeta^{i+1} \mu_{i,j} \mu_{i+1,j+1}a^{i+2}_{j+2}) + \zeta\lambda_{i, j}\mu_{i,j+1} a^{i+1}_{j+2}.
\]
Thus we see that $\lambda_{i+1,j+1} \mu_{i,j} = \zeta \lambda_{i,j}\mu_{i,j+1}$, which is the third condition on the scalars.
We also obtain that the relation $x^n=0$ applied to $a^i_j$ imposes no further restrictions on the $x$-action on $a^i_j$; this is verified in Lemma~\ref{lem:xact0A} of the appendix.
\end{proof}

\subsection{Taft actions on Type B $\mathbb{Z}_n$-minimal quivers}

In this subsection, let $Q$ be a $\mathbb{Z}_n$-minimal quiver of Type $B$.  By Definition~\ref{def:typeAB}, we know that $Q$ is a subquiver of $K_{m, m'}$ for some positive integers $m, m'$ both dividing $n$.  Recall that the path algebra $K_{m, m'}$ has basis $e_{i_+}$, $e_{j_-}$, $b^i_j$ where $1 \leq i \leq m$ and $1 \leq j \leq m'$.  

\begin{theorem} \label{thm:T(n)typeB} Retain the notation above.
Any $T(n)$-action on the path algebra of  Type $B$ $\mathbb{Z}_n$-minimal quiver $Q$ is given by

\begin{equation} \label{eq:B}
x\cdot b^i_j ~=~
(\gamma_-)\zeta^j b^i_j - (\gamma_+)\mu_{i,j}\zeta^{i+1} b^{i+1}_{j+1} + \lambda_{i,j} b^i_{j+1}
\end{equation}
where $\gamma_+, \gamma_- \in \kk$ from \eqref{eq:xonvertex2}, $\mu_{i,j} \in \kk^{\times}$ from \eqref{eq:prodmu}, and $\lambda_{i, j} \in \kk$ are all scalars satisfying 
\begin{itemize}
\item $\prod_{\ell=0}^{n-1} \mu_{i +\ell, j + \ell} =1$;
\item $\lambda_{i,j} = 0$ if the arrow $b^i_{j+1}$ does not exist in $Q$;
\item $\lambda_{i+1,j+1}\mu_{i,j} = \zeta \lambda_{i,j}\mu_{i,j+1}$;
\item $(\gamma_+)^n = (\gamma_-)^n + \prod_{\ell = 0}^{n-1} \lambda_{i, j+\ell}$.
\end{itemize} 
The superindices of arrows are taken modulo $m$ and the subindices of arrows are taken modulo $m'$.
\end{theorem}

\begin{proof}
The formula \eqref{eq:B} and first three conditions on the scalars are derived exactly as in the proof of Theorem~\ref{thm:T(n)typeA}, replacing $a^{\star}_{\star}$ with $b^{\star}_{\star}$ and $\gamma$ with $\gamma_{\pm}$, appropriately.  It just remains to check that $x^n$ acts by 0; this is equivalent to the last condition on the scalars, as shown in Lemma \ref{lem:xact0B}.
\end{proof}

\section{Gluing Taft algebra actions on minimal quivers} \label{sec:gluing}
 
\subsection{Gluing actions from components}
In this section, we provide a recipe for gluing actions of Taft algebras on minimal quivers. We also show that, given any quiver with $\mathbb{Z}_n$-symmetry, one can construct an inner faithful extended action of $T(n)$ on the path algebra of this quiver.

To begin, consider the following definition.

\begin{definition}[$\mathbb{Z}_n$-component] \label{def:component}
Let $Q$ be a quiver with an action of $\mathbb{Z}_n$, and consider the set of $\mathbb{Z}_n$-minimal subquivers of $Q$, partially ordered by inclusion.
We say that a $\mathbb{Z}_n$-minimal subquiver of $Q$ is a {\it $\mathbb{Z}_n$-component of $Q$} if it is maximal in the given ordering.
\end{definition}

\begin{lemma} \label{lem:Gminl}
Fix an action of  $\mathbb{Z}_n$ on a quiver $Q$. Then, there  exists a collection of $\mathbb{Z}_n$-components of $Q$, unique up to relabeling, such that $Q$ is obtained by gluing this collection.
\end{lemma}
\begin{proof}The $\mathbb{Z}_n$-components of $Q$ exist and are uniquely determined by the definition because they are the maximal elements of a finite poset.  Each arrow of $Q$ lies in some $\mathbb{Z}_n$-minimal subquiver, so the $\mathbb{Z}_n$-components cover $Q$. So, it suffices to show that the intersection of two distinct $\mathbb{Z}_n$-components consists entirely of vertices.  If $Q^1$ and $Q^2$ are two $\mathbb{Z}_n$-minimal subquivers such that $Q^1 \cap Q^2$ contains an arrow, then we can see from the definition of minimality that $Q^1$ and $Q^2$ have the same set of vertices. 
Thus, $Q^1 \cup Q^2$ is a $\mathbb{Z}_n$-minimal subquiver of $Q$.  Repeat this process to conclude that, by maximality, any two distinct $\mathbb{Z}_n$-components can only have vertices in their intersection.
\end{proof}

A visualization of the result above can be found in \textsc{Step 1} of Examples~\ref{ex:gluing} and~\ref{ex:gluing3} below.

\begin{lemma}\label{lem:restrict}
Any $T(n)$-action on a path algebra $\kk Q$ restricts to an action on the path algebra of each $\mathbb{Z}_n$-component of $Q$.
\end{lemma}
\begin{proof}
Let $Q^i$ be a $\mathbb{Z}_n$-component of $Q$.  Since $Q^i$ is $\mathbb{Z}_n$-minimal, by definition $\kk Q^i$ is stable under the action of $g$, so it suffices to show that $\kk Q^i$ is stable under the action of $x$.  From Proposition \ref{prop:Taftactarrow}, it is enough to see that $\sigma(a) \in Q^i$ when $a \in Q^i$. Suppose not, that is, $\sigma(a) \in Q^j$ for some $j \neq i$. Then, $Q^i \cup Q^j$ is a $\mathbb{Z}_n$-minimal quiver, which contradicts the maximality of the $\mathbb{Z}_n$-component $Q^i$.
\end{proof}

\begin{definition}[Compatibility] \label{def:compatible}
Let $Q$ be a quiver and $Q^1, \dotsc, Q^r \subseteq Q$ a collection of subquivers of $Q$.  Suppose that we have a $T(n)$-action on each path algebra $\kk Q^i$.  We say that this collection of $T(n)$-actions is \emph{compatible} if, for each pair $(i,j)$, the restriction of the actions on $\kk Q^i$ and  on $\kk Q^j$ to $\kk[Q^i \cap Q^j]$ are the same. 
\end{definition}

Now we have our main result of the section.

\begin{theorem} \label{thm:glue}
Let $Q$ be a quiver with $\mathbb{Z}_n$-action.  The $T(n)$-actions on the path algebra of $Q$ extending the given $\mathbb{Z}_n$-action are in bijection with compatible collections of $T(n)$-actions on path algebras of the $\mathbb{Z}_n$-components of $Q$.
\end{theorem}
\begin{proof}
Given a $T(n)$-action on $\kk Q$, it restricts to a $T(n)$-action on each path algebra of a $\mathbb{Z}_n$-component of $Q$ by Lemma \ref{lem:restrict}.
On the other hand, suppose we have a collection of compatible $T(n)$-actions on the path algebras of the $\mathbb{Z}_n$-components of $Q$.
This uniquely determines an action of $T(n)$ on $\kk Q$  as follows:
\begin{itemize}
\item The action on each arrow of $\kk Q$ is uniquely determined since each arrow lies in a unique $\mathbb{Z}_n$-component of $Q$;
\item The action on any vertex is determined because every vertex lies in at least one $\mathbb{Z}_n$-component;
\item Suppose that a vertex lies in multiple $\mathbb{Z}_n$-components. Then, the action on this vertex is uniquely determined because the actions on different $\mathbb{Z}_n$-components restrict to the same action on vertices in their intersection, due to compatibility.
\end{itemize}
\vspace{-.2in}
\end{proof}

\begin{corollary}\label{cor:non-trivial}
If a quiver $Q$ admits a faithful $\mathbb{Z}_n$-action, then $\kk Q$ admits an inner faithful action of the Taft algebra $T(n)$, extending the given $\mathbb{Z}_n$-action on $Q$.
\end{corollary}

\begin{proof}
This follows immediately from Theorem \ref{thm:glue} and Proposition \ref{prop:TaftactQ0} since $Q$ admits an orbit of vertices of size $n$ in this case.
\end{proof}

\begin{example}
This example shows that a faithful $\mathbb{Z}_n$-action on $Q$ is not necessary for $\kk Q$ to admit an inner faithful action of $T(n)$.  Fix $\zeta$, a primitive fourth root of unity. Consider the action of $T(4)$ on $\kk K_2$ (see {\sf (I)} of Figure~\ref{fig:Z2minimal}) 
given by
\begin{equation*}
g\cdot e_i = e_{i+1}, \quad \quad g \cdot a^i_j = \zeta a^{i+1}_{j+1}, \quad  \quad x\cdot e_i = 0, \quad \quad x\cdot a^i_j = \lambda e_i,
\end{equation*}
for $i \neq j$, where sub/superscripts are taken modulo 2, and $\lambda \in \kk^\times$ is an arbitrary nonzero scalar.
 By Theorem~\ref{thm:T(n)typeA}, this defines an action of $T(4)$ on the path algebra $\kk K_2$, which is inner faithful even though the induced action of $\mathbb{Z}_4$ on the quiver $K_2$ is not faithful.
\end{example}

\subsection{Algorithm to explicitly parametrize $T(n)$-actions on $\kk Q$} \label{sec:algorithm}
Let $Q$ be a quiver that admits an action of $\mathbb{Z}_n$. We construct  all actions of $T(n)$ on $\kk Q$ which extend the given $\mathbb{Z}_n$-action via the following steps.  Those for which $x$ does not act by 0 are inner faithful by Lemma~\ref{lem:Taftfaithful}.  The reader may wish to refer to one of the examples in the next subsection for an illustration of the algorithm below.

First, we know that 
$Q$ decomposes uniquely into the union of certain $\mathbb{Z}_n$-minimal quivers $\{Q^\ell\}_{\ell=1}^r$ (namely, $\mathbb{Z}_n$-components) so that $Q = Q^1 \circledast \cdots \circledast Q^r$, due to Lemma~\ref{lem:Gminl}.
\medskip

\begin{center}
\framebox{
\begin{tabular}{c} 
\textsc{\underline{Step 1}}
\medskip\\ 
Decompose $Q$ into this unique union of $\mathbb{Z}_n$-components $\{Q^\ell\}$.
\end{tabular}}
\end{center}
\medskip

\noindent Next, we define the extended action of $T(n)$ on the path algebra of each component $Q^\ell$.  Let $m, m'$ be positive divisors of $n$. Recall that Type A (resp., Type B) $\mathbb{Z}_n$-minimal quivers are subquivers of $K_m$ with $m >1$ (resp., of $K_{m,m'}$).
\medskip

\begin{center} 
\framebox{
\begin{tabular}{c}
\textsc{\underline{Step 2}}
\medskip\\ 
For each $\mathbb{Z}_n$-component $Q^\ell$ of Type A, label its vertices by $\{1^{(\ell)}, \dots, m^{(\ell)}\}$.\\
For each $\mathbb{Z}_n$-component $Q^\ell$ of Type B, label its  sources and sinks by\\
$\{1_+^{(\ell)}, \dots, m_+^{(\ell)}\}$ and $\{1_-^{(\ell)}, \dots (m')_-^{(\ell)}\}$, respectively.
\end{tabular}
}
\end{center}
\medskip

\noindent Recall that $\zeta$ is the primitive $n$-th root of unity from the definition of $T(n)$.  Now invoke Proposition~\ref{prop:TaftactQ0} in the following step.
\medskip

\begin{center}
 \framebox{
\begin{tabular}{c}
\textsc{\underline{Step 3}}
\medskip\\ 
Take scalars $\gamma^{(\ell)}, \gamma_+^{(\ell)}, \gamma_-^{(\ell)} \in \kk$ and define
 
\medskip\\

\begin{tabular}{lll}
$x \cdot e_{i^{(\ell)}}$ &$= \gamma^{(\ell)} \zeta^i(e_{i^{(\ell)}} - \zeta e_{(i+1)^{(\ell)}})$ &$\quad \quad \text{for Type A}$\\
$x \cdot e_{i^{(\ell)}_+}$ &$= \gamma_+^{(\ell)} \zeta^i(e_{i^{(\ell)}_+} - \zeta e_{(i+1)^{(\ell)}_+})$ &$\quad \quad \text{for Type B}$\\
$x \cdot e_{i^{(\ell)}_-}$ &$= \gamma_-^{(\ell)} \zeta^i(e_{i^{(\ell)}_-} - \zeta e_{(i+1)^{(\ell)}_-})$ &$\quad \quad \text{for Type B}$
\end{tabular}
\medskip\\

where the indices are taken modulo $m$ for Type A, and are taken modulo $m$ or $m'$ for Type B.\\
Here, $\gamma^{(\ell)}$ = 0 (resp., $\gamma_+^{(\ell)} = 0$ or $\gamma_-^{(\ell)} = 0$) if $m <n$ for Type A (resp., $m <n$ or $m'<n$ for Type B).\\
 To obtain compatibility, impose relations amongst the scalars by identifying vertices.
\end{tabular}
}
\end{center}
\medskip

\noindent Finally, we invoke Theorems~\ref{thm:T(n)typeA} and~\ref{thm:T(n)typeB} to get all actions of $T(n)$ on the arrows of $Q$.
\medskip

\begin{center} 
\framebox{
\begin{tabular}{c}
\textsc{\underline{Step 4}}
\medskip\\
Label the arrows of each Type A (resp., Type B) component of $Q$ by arrows $(a^i_j)^{(\ell)}$ (resp., $(b^i_j)^{(\ell)}$).
\end{tabular}}
\end{center}
\medskip

\begin{center}
 \framebox{
\begin{tabular}{c}
\textsc{\underline{Step 5}}
\medskip\\
Given the scalars  $\gamma^{(\ell)}, \gamma_+^{(\ell)}, \gamma_-^{(\ell)} \in \kk$ of \textsc{Step 3}, we have, for all arrows $a^i_j$ of\\ a component of Type A and $b^i_j$ of a component of Type B, that:
\medskip\\
\begin{tabular}{rl}
$x\cdot (a^i_j)^{(\ell)}$ &$= \gamma^{(\ell)}\left(\zeta^j (a^i_j)^{(\ell)} - \zeta^{i+1}\mu_{i,j}^{(\ell)} (a^{i+1}_{j+1})^{(\ell)}\right) + \lambda_{i, j}^{(\ell)} (a^i_{j+1})^{(\ell)}$\\\\
$x\cdot (b^i_j)^{(\ell)}$ &$=
\gamma_-^{(\ell)}\zeta^j (b^i_j)^{(\ell)} - \gamma_+^{(\ell)}\zeta^{i+1}\mu_{i,j}^{(\ell)} (b^{i+1}_{j+1})^{(\ell)} + \lambda_{i,j}^{(\ell)} (b^i_{j+1})^{(\ell)}$
\end{tabular}
\medskip\\
where $\mu_{i,j}^{(\ell)}$, $\lambda_{i, j}^{(\ell)} \in \kk$ are scalars satisfying the conditions of Theorems~\ref{thm:T(n)typeA} and~\ref{thm:T(n)typeB}.
\end{tabular}
}
\end{center}
\medskip

\noindent Thus, the desired extended action of the $n$-th Taft algebra on $\kk Q$ is complete by the five steps above.

\subsection{Examples}  In this section, we illustrate the algorithm of the previous section to get  actions of $T(n)$ on $\kk Q$ for various quivers $Q$ having $\mathbb{Z}_n$ symmetry. 
\medskip
\begin{center}
{\it For ease of exposition, we take all parameters $\mu_{i,j}$ equal to 1 in this section.}
\end{center}
\smallskip

\begin{example}\label{ex:gluing}
Consider the following quiver $Q$: 

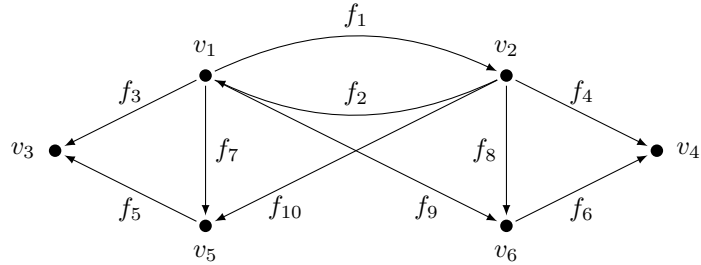
\[
\vcenter{\hbox{\begin{tikzpicture}[point/.style={shape=circle,fill=black,scale=.5pt,outer sep=3pt},>=latex]
   \node[point, label={above:$v_1$}] (1) at (-2,2) {};
   \node[point, label={above:$v_2$}] (2) at (2,2) {};
   \node[point, label={left: $v_3$}] (3) at (-4,1) {};
   \node[point, label={below:$v_5$}] (4) at (-2,0) {};
 \node[point, label={below:$v_6$}] (5) at (2,0) {};
   \node[point, label={right:$v_4$}] (6) at (4,1) {};
\draw[->] (1) edge  node[above]{$f_3$} (3);
\draw[->] (4) edge  node[below]{$f_5$} (3);
\draw[->] (2) edge  node[above]{$f_4$} (6);
\draw[->] (5) edge  node[below]{$f_6$} (6);
  \path[->]
  	(1) edge[bend left=25] node[above] {$f_1$} (2) 
  	(2) edge[bend left=25] node[above] {$f_2$} (1);
\path[->] (1) edge  node[right]{$f_7$} (4)
	(2) edge  node[left]{$f_8$} (5)
	(1) edge node[below,pos=0.75]{$f_9$} (5)
	(2) edge node[below,pos=0.75]{$f_{10}$} (4);
   \end{tikzpicture}}}
\]
\medskip

\noindent which has $\mathbb{Z}_2$-symmetry by reflection over the central vertical axis and exchanging arrows $f_1$ and $f_2$. So, we decompose $Q$ into the four $\mathbb{Z}_2$-components as follows. Here, the dotted red arrows indicate the action of $\mathbb{Z}_2$ on $Q_0$.

\pagebreak

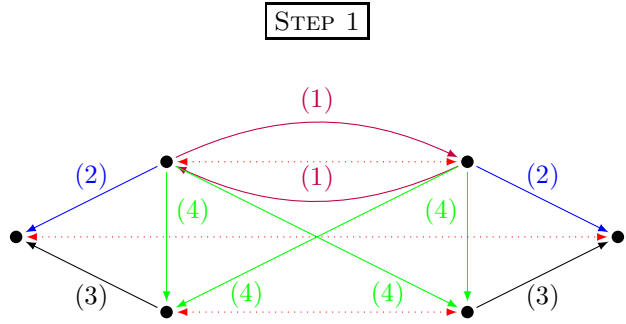
\[
\text{\framebox{\textsc{Step 1}}}
\]
\[
 \vcenter{\hbox{\begin{tikzpicture}[point/.style={shape=circle,fill=black,scale=.5pt,outer sep=3pt},>=latex]
   \node[point] (1) at (-2,2) {};
   \node[point] (2) at (2,2) {};
   \node[point] (3) at (-4,1) {};
   \node[point] (4) at (-2,0) {};
 \node[point] (5) at (2,0) {};
   \node[point] (6) at (4,1) {};
  \path[->,purple]
  	(1) edge[bend left=25] node[above] {$(1)$} (2) 
  	(2) edge[bend left=25] node[above] {$(1)$} (1);
\draw[->][blue] (1) edge node[above]{$(2)$} (3);
\draw[->] (4) edge node[below]{$(3)$}(3);
\draw[->][blue] (2) edge node[above]{$(2)$}(6);
\draw[->] (5) edge node[below]{$(3)$}(6);
\path[->,green] (1) edge  node[right,pos=.3]{$(4)$} (4)
	(2) edge  node[left, pos=.3]{$(4)$} (5)
	(1) edge node[below,pos=0.75]{$(4)$} (5)
	(2) edge node[below,pos=0.75]{$(4)$} (4);
\draw[<->][dotted,red] (1) edge (2);
\draw[<->][dotted,red] (4) edge (5);
\draw[<->][dotted,red] (3) edge (6);
   \end{tikzpicture}}}
\]
\medskip

Now choose scalars $\gamma^{(1)}$, $\gamma^{(2)}_+$, $\gamma^{(2)}_-$, $\gamma^{(3)}_+$, $\gamma^{(3)}_-, \gamma^{(4)}_+$, $\gamma^{(4)}_-  \in \kk$ to execute Steps 2 and 3. 
\medskip


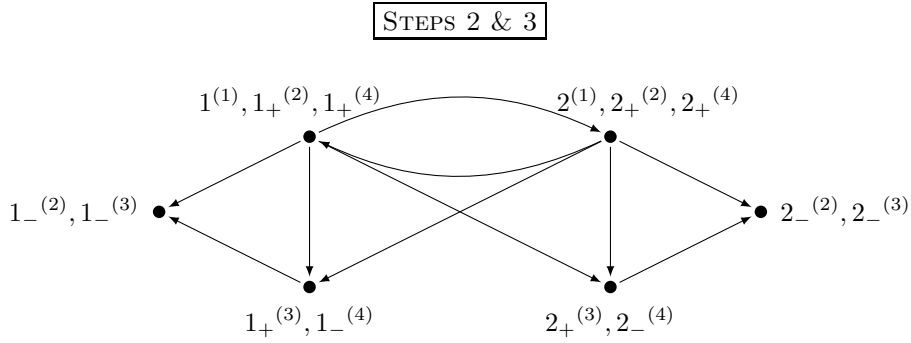
\[
\framebox{\textsc{Steps 2 \& 3}}
\]
\[
\vcenter{\hbox{\begin{tikzpicture}[point/.style={shape=circle,fill=black,scale=.5pt,outer sep=3pt},>=latex]
   \node[point, label={above:$1^{(1)}, {1_+}^{(2)}, {1_+}^{(4)}\hspace{.2in}$}] (1) at (-2,2) {};
   \node[point, label={above:$\hspace{.4in}2^{(1)}, {2_+}^{(2)}, {2_+}^{(4)}$}] (2) at (2,2) {};
   \node[point, label={left: ${1_-}^{(2)}, {1_-}^{(3)}$}] (3) at (-4,1) {};
   \node[point, label={below:${1_+}^{(3)}, {1_-}^{(4)}$}] (4) at (-2,0) {};
 \node[point, label={below:${2_+}^{(3)}, {2_-}^{(4)}$}] (5) at (2,0) {};
   \node[point, label={right:${2_-}^{(2)}, {2_-}^{(3)}$}] (6) at (4,1) {};
\draw[->] (1) edge   (3);
\draw[->] (4) edge  (3);
\draw[->] (2) edge  (6);
\draw[->] (5) edge  (6);
  \path[->]
  	(1) edge[bend left=25] node[above] {} (2) 
  	(2) edge[bend left=25] node[above] {} (1);
\path[->] (1) edge  node[right]{} (4)
	(2) edge  node[left]{} (5)
	(1) edge node[below,pos=0.75]{} (5)
	(2) edge node[below,pos=0.75]{} (4);
   \end{tikzpicture}}}
\]
\vspace{.2in}

\noindent We have that:
\[
x\cdot e_1^{(1)} = -\gamma^{(1)}(e_1^{(1)} + e_2^{(1)}) \qquad x\cdot e_{2}^{(1)} = \gamma^{(1)}(e_2^{(1)} + e_1^{(1)})
\]
while, for $\ell = 2,3,4$,
\[\begin{array}{rlllrl}
x \cdot e_{1_+}^{(\ell)} &= -\gamma_+^{(\ell)}\left(e_{1_+}^{(\ell)}+e_{2_+}^{(\ell)}\right) &&
x \cdot e_{2_+}^{(\ell)} &= \gamma_+^{(\ell)}\left(e_{2_+}^{(\ell)}+e_{1_+}^{(\ell)}\right)\\
x \cdot e_{1_-}^{(\ell)} &= -\gamma_-^{(\ell)}\left(e_{1_-}^{(\ell)}+e_{2_-}^{(\ell)}\right)&&
x \cdot e_{2_-}^{(\ell)} &= \gamma_-^{(\ell)}\left(e_{2_-}^{(\ell)}+e_{1_-}^{(\ell)}\right).
\end{array}
\]
\noindent By using the identification of vertices:
\[
\begin{array}{lllll}
v_1:=  1^{(1)}= {1_+}^{(2)}= {1_+}^{(4)}&&
v_2:=  2^{(1)}= {2_+}^{(2)}={2_+}^{(4)}\\
v_3:= {1_-}^{(2)} = {1_-}^{(3)}&&
v_4:= {2_-}^{(2)} = {2_-}^{(3)}\\
v_5:=  {1_+}^{(3)} = {1_-}^{(4)}&&
v_6:=  {2_+}^{(3)} = {2_-}^{(4)},
\end{array}
\]

\noindent we have that

$$\gamma :=  \gamma^{(1)}= \gamma_+^{(2)} = \gamma_+^{(4)}  \quad \quad
\gamma' := \gamma_-^{(2)} = \gamma_-^{(3)} \quad \quad
\gamma'' := \gamma_+^{(3)} =  \gamma_-^{(4)}.$$

\medskip

 Now, we need to label the arrows of $Q$ appropriately, and invoke Theorems~\ref{thm:T(n)typeA} and \ref{thm:T(n)typeB} to get the desired action of $T(2)$ on $\kk Q_1$.

\pagebreak

\[
\text{\framebox{\textsc{Steps 4 \& 5}}}
\]
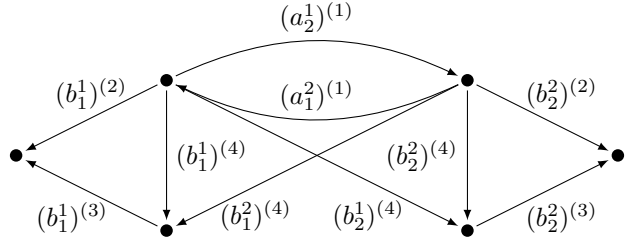
\[
\vcenter{\hbox{\begin{tikzpicture}[point/.style={shape=circle,fill=black,scale=.5pt,outer sep=3pt},>=latex]
   \node[point] (1) at (-2,2) {};
   \node[point] (2) at (2,2) {};
   \node[point] (3) at (-4,1) {};
   \node[point] (4) at (-2,0) {};
 \node[point] (5) at (2,0) {};
   \node[point] (6) at (4,1) {};
\draw[->] (1) edge  node[above]{$(b^1_1)^{(2)}$} (3);
\draw[->] (4) edge  node[below]{$(b^1_1)^{(3)}\hspace{.2in}$} (3);
\draw[->] (2) edge  node[above]{$\hspace{.2in}(b^2_2)^{(2)}$} (6);
\draw[->] (5) edge  node[below]{$\hspace{.2in}(b^2_2)^{(3)}$} (6);
  \path[->]
  	(1) edge[bend left=25] node[above] {$(a^1_2)^{(1)}$} (2) 
  	(2) edge[bend left=25] node[above] {$(a^2_1)^{(1)}$} (1);
\path[->] (1) edge  node[right]{$(b^1_1)^{(4)}$} (4)
	(2) edge  node[left]{$(b^2_2)^{(4)}$} (5)
	(1) edge node[below,pos=0.75]{$(b^1_2)^{(4)}\hspace{.2in}$} (5)
	(2) edge node[below,pos=0.75]{$\hspace{.1in}(b^2_1)^{(4)}$} (4);
   \end{tikzpicture}}}
\]
\medskip

\[
\begin{split}
x\cdot (a^i_j)^{(1)} &= \gamma^{(1)}\left(\zeta^j (a^i_j)^{(1)} - \zeta^{i+1} (a^{i+1}_{j+1})^{(1)}\right) + \lambda_{i, j}^{(1)} (a^i_{j+1})^{(1)}\\
x\cdot (b^i_j)^{(\ell)} &=
\gamma_-^{(\ell)}\zeta^j (b^i_j)^{(\ell)} - \gamma_+^{(\ell)}\zeta^{i+1} (b^{i+1}_{j+1})^{(\ell)} + \lambda_{i,j}^{(\ell)} (b^i_{j+1})^{(\ell)}, \quad \text{$\ell$ = 2,3,4}.
\end{split}
\]

\medskip

\noindent  Using the conditions on the scalars from Theorems~\ref{thm:T(n)typeA} and~\ref{thm:T(n)typeB}, we find that the $x$-action on the arrows is controlled by additional parameters $\lambda:=\lambda^{(1)}_{1,2}$ as $\sigma(a)$ exists for component (1), and $\lambda':=\lambda^{(4)}_{1,1}, \lambda'':=\lambda^{(4)}_{1,2}$ as $\sigma(a)$ exists for component (4). Putting this all together, the action of $\mathbb{Z}_2$ on $Q$ (where by abuse of notation, $v_i$ denotes the trivial path at $v_i$) given by

\[
\begin{array}{|llllllll|}
\hline
g \cdot v_1 = v_2 && g \cdot v_2 = v_1&&&
g \cdot f_1 = f_2 && g \cdot f_2 = f_1\\
g \cdot v_3 = v_4 && g \cdot v_4 = v_3&&&
g \cdot f_3 = f_4 && g \cdot f_4 = f_3\\
g \cdot v_5 = v_6 && g \cdot v_6 = v_5&&&
g \cdot f_5 = f_6 && g \cdot f_6 = f_5\\
&&&&&g \cdot f_7 = f_8 && g \cdot f_8 = f_7\\
&&&&&g \cdot f_9 = f_{10} && g \cdot f_{10} = f_9\\
\hline
\end{array}
\]
\medskip

\noindent extends to an action of the Sweedler algebra on $\kk Q$, as follows:
\[
\begin{array}{|lccl|}
\hline
x \cdot v_1 = - \gamma (v_1 + v_2) &&& x \cdot v_2 =  \gamma (v_2 + v_1)\\
x \cdot v_3 = - \gamma' (v_3 + v_4) &&& x \cdot v_4 =  \gamma' (v_4 + v_3)\\
x \cdot v_5 = - \gamma'' (v_5 + v_6) &&& x \cdot v_6 =  \gamma'' (v_6 + v_5)\\
\hline
\end{array}
\]
\[
\begin{array}{|lccl|}
\hline
x \cdot f_1 = \gamma f_1 -  {\gamma} f_2 + \zl e_1&&& 
x \cdot f_2 = \gamma f_1 -  {\gamma} f_2 - \zl e_2 \\
x \cdot f_3 = -\gamma' f_3 -  {\gamma} f_4 &&& 
x \cdot f_4 =  \gamma' f_4 +  {\gamma} f_3\\
x \cdot f_5 = - \gamma' f_5 -  {\gamma''} f_6 &&& 
x \cdot f_6 =  \gamma' f_6 +  {\gamma''} f_5\\
x \cdot f_7 = - \gamma'' f_7 -  {\gamma} f_8 + \zl' f_9&&& 
x \cdot f_8 =  \gamma'' f_8 +  {\gamma} f_7 - \zl' f_{10} \\
x \cdot f_9 =  \gamma'' f_9 -  {\gamma} f_{10} + \zl'' f_7&&& 
x \cdot f_{10} = - \gamma'' f_{10} +  {\gamma} f_9 - \zl'' f_8. \\
\hline
\end{array}
\]
\[
\begin{array}{|c|}
\hline
\text{for any scalars }  \gamma, \gamma', \gamma'', \zl, \zl', \zl'' \in \kk \\
\text{which satisfy }  (\zg)^2 = (\zg'')^2 + \zl' \zl''\\
\hline
\end{array}
\]

\medskip

Note that the Sweedler algebra action restricted to the path algebras of respective components (1), (2), (3), (4) is the same as the Sweedler algebra action on the path algebra of respective quivers {\sf(I)}, {\sf(V)}, {\sf(V)}, {\sf(VI)} from Theorem~\ref{thm:T(2)minimal}, as expected; namely, $Q= Q_{\sf(I)} \circledast Q_{\sf(V)} \circledast Q_{\sf(V)}\circledast  Q_{\sf(VI)}$.
\medskip

\end{example}

\begin{example}\label{ex:gluing3}
Consider the following quiver $Q$: 

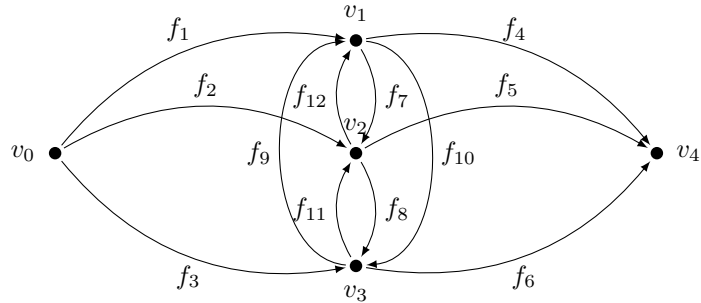
\[
\vcenter{\hbox{\begin{tikzpicture}[point/.style={shape=circle,fill=black,scale=.5pt,outer sep=3pt},>=latex]
\node[point, label={left:$v_0$}] (0) at (-4,0) {};
\node[point, label={above:$v_1$}] (1) at (0,1.5) {};
\node[point, label={above:$v_2$}] (2) at (0,0) {};
\node[point, label={below:$v_3$}] (3) at (0,-1.5) {};
\node[point, label={right:$v_4$}] (4) at (4,0) {};
  \path[->]
      (0) edge[bend left=30] node[above] {$f_1\hspace{.1in}$} (1) 
      (0) edge[bend left=30] node[above] {$f_2$} (2) 
      (0) edge[bend right=30] node[below] {$f_3$} (3) 
  	(1) edge[bend left=30] node[right] {$f_7$} (2) 
  	(2) edge[bend left=30] node[left] {$f_{12}$} (1)
 	(2) edge[bend left=30] node[right] {$f_8$} (3) 
  	(3) edge[bend left=30] node[left] {$f_{11}$} (2)
 	(1) edge[bend left=85] node[right] {$f_{10}$} (3) 
  	(3) edge[bend left=85] node[left] {$f_{9}$} (1)
      (1) edge[bend left=30] node[above] {$f_4\hspace{.1in}$} (4) 
      (2) edge[bend left=30] node[above] {$f_5$} (4) 
      (3) edge[bend right=30] node[below] {$f_6$} (4) ;
   \end{tikzpicture}}}
\]
\medskip

\noindent which has $\mathbb{Z}_3$-symmetry. We decompose $Q$ into three $\mathbb{Z}_3$-components as follows. Here, the dotted red arrows indicate the action of $\mathbb{Z}_3$ on $Q_0$.

\[
\text{\framebox{\textsc{Step 1}}}
\]
\[
 \vcenter{\hbox{\begin{tikzpicture}[point/.style={shape=circle,fill=black,scale=.5pt,outer sep=3pt},>=latex]
\node[point] (0) at (-4,0) {};
\node[point] (1) at (0,1.5) {};
\node[point] (2) at (0,0) {};
\node[point] (3) at (0,-1.5) {};
\node[point] (4) at (4,0) {};
  \path[->,purple]
      (0) edge[bend left=30] node[above] {$(1)\hspace{.1in}$} (1) 
      (0) edge[bend left=30] node[above] {$(1)$} (2) 
      (0) edge[bend right=30] node[below] {$(1)$} (3) ;
  \path[->,black]
  	(1) edge[bend left=30] node[right] {$(3)$} (2) 
  	(2) edge[bend left=30] node[left] {$(3)$} (1)
 	(2) edge[bend left=30] node[right] {$(3)$} (3) 
  	(3) edge[bend left=30] node[left] {$(3)$} (2)
 	(1) edge[bend left=85] node[right] {$(3)$} (3) 
  	(3) edge[bend left=85] node[left] {$(3)$} (1);
  \path[->,blue]
      (1) edge[bend left=30] node[above] {$(2)\hspace{.1in}$} (4) 
      (2) edge[bend left=30] node[above] {$(2)$} (4) 
      (3) edge[bend right=30] node[below] {$(2)$} (4) ;
\draw[->][dotted,red] (0) to [out=220,in=140,looseness=10] (0);
\draw[->][dotted,red] (4) to [out=400,in=320,looseness=10] (4);
\draw[->][dotted,red] (1) to  (2);
\draw[->][dotted,red] (2) to  (3);
\draw[->][dotted,red] (3) to [bend left=50] (1);
   \end{tikzpicture}}}
\]
\medskip

\noindent Now choose  scalars $\gamma_+^{(1)}$, $\gamma_-^{(1)}$,$\gamma^{(2)}_+$, $\gamma^{(2)}_-$, $\gamma^{(3)} \in \kk$  to execute Steps 2 and 3. Further, let $\omega$ be a primitive third root of unity.
\medskip

\[
\framebox{\textsc{Steps 2 \& 3}}
\]
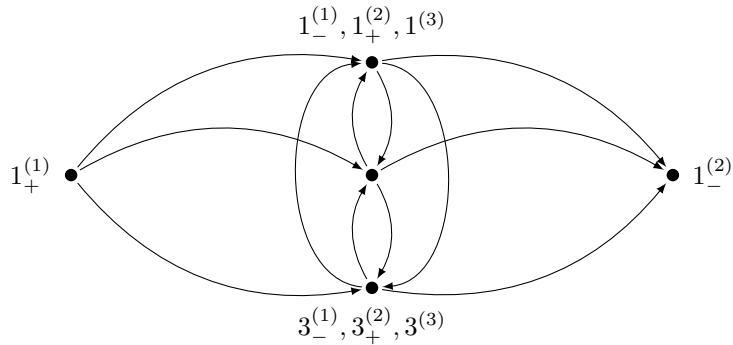
\[
\vcenter{\hbox{\begin{tikzpicture}[point/.style={shape=circle,fill=black,scale=.5pt,outer sep=3pt},>=latex]
\node[point, label={left:$1_+^{(1)}$}] (0) at (-4,0) {};
\node[point, label={above:$1_-^{(1)}, 1_+^{(2)}, 1^{(3)}$}] (1) at (0,1.5) {};
\node[point] (2) at (0,0) {};
\node[point, label={below:$3_-^{(1)}, 3_+^{(2)}, 3^{(3)}$}] (3) at (0,-1.5) {};
\node[point, label={right:$1_-^{(2)}$}] (4) at (4,0) {};
  \path[->]
      (0) edge[bend left=30] (1) 
      (0) edge[bend left=30]  (2) 
      (0) edge[bend right=30]  (3) 
  	(1) edge[bend left=30]  (2) 
  	(2) edge[bend left=30]  (1)
 	(2) edge[bend left=30] (3) 
  	(3) edge[bend left=30] (2)
 	(1) edge[bend left=85]  (3) 
  	(3) edge[bend left=85]  (1)
      (1) edge[bend left=30] (4) 
      (2) edge[bend left=30]  (4) 
      (3) edge[bend right=30]  (4) ;
   \end{tikzpicture}}}
\]
where the middle vertex is labelled by $2_-^{(1)},~2_+^{(2)}, ~2^{(3)}$.

\vspace{.2in}

\noindent We have that $x \cdot e_{1_+}^{(1)} =x \cdot e_{1_-}^{(2)} =0$ since these vertices are fixed by $g$, and furthermore:
\[
\begin{array}{l}
x \cdot e_{1_-}^{(1)} = \gamma_-^{(1)} \omega \left(e_{1_-}^{(1)} - \omega e_{2_-}^{(1)}\right)\\
x \cdot e_{2_-}^{(1)} = \gamma_-^{(1)} \omega^2 \left(e_{2_-}^{(1)} - \omega e_{3_-}^{(1)}\right)\\
x \cdot e_{3_-}^{(1)} = \gamma_-^{(1)}  \left(e_{3_-}^{(1)} - \omega e_{1_-}^{(1)}\right)
\end{array}
\quad
\begin{array}{l}
x \cdot e_{1_+}^{(2)} = \gamma_+^{(2)} \omega \left(e_{1_+}^{(2)} - \omega e_{2_+}^{(2)}\right)\\
x \cdot e_{2_+}^{(2)} = \gamma_+^{(2)} \omega^2 \left(e_{2_+}^{(2)} - \omega e_{3_+}^{(2)}\right)\\
x \cdot e_{3_+}^{(2)} = \gamma_+^{(2)}  \left(e_{3_+}^{(2)} - \omega e_{1_+}^{(2)}\right)
\end{array}
\quad
\begin{array}{l}
x \cdot e_{1}^{(3)} = \gamma^{(3)} \omega \left(e_{1}^{(3)} - \omega e_{2}^{(3)}\right)\\
x \cdot e_{2}^{(3)} = \gamma^{(3)} \omega^2 \left(e_{2}^{(3)} - \omega e_{3}^{(3)}\right)\\
x \cdot e_{3}^{(3)} = \gamma^{(3)}  \left(e_{3}^{(3)} - \omega e_{1}^{(3)}\right).
\end{array}
\]

\noindent By using the identification of vertices,
\[
v_1 = 1_-^{(1)} = 1_+^{(2)} = 1^{(3)} \quad \quad 
v_2 = 2_-^{(1)} = 2_+^{(2)} = 2^{(3)} \quad \quad
v_3 = 3_-^{(1)} = 3_+^{(2)} = 3^{(3)},
\]

\noindent we have that

$$\gamma :=  \gamma_-^{(1)} = \gamma_+^{(2)} = \gamma^{(3)},
\quad \quad
\gamma_+^{(1)} = 0, \quad \quad \gamma_-^{(2)} = 0.$$

\medskip

 Now, we need to label the arrows of $Q$ appropriately, and invoke Theorems~\ref{thm:T(n)typeA} and \ref{thm:T(n)typeB} to get the desired action of $T(3)$ on $\kk Q_1$.

\medskip

\[
\text{\framebox{\textsc{Steps 4 \& 5}}}
\]
\[
 \vcenter{\hbox{\begin{tikzpicture}[point/.style={shape=circle,fill=black,scale=.5pt,outer sep=3pt},>=latex]
\node[point] (0) at (-4,0) {};
\node[point] (1) at (0,1.5) {};
\node[point] (2) at (0,0) {};
\node[point] (3) at (0,-1.5) {};
\node[point] (4) at (4,0) {};
  \path[->]
      (0) edge[bend left=30] node[above] {$(b^1_1)^{(1)}\hspace{.1in}$} (1) 
      (0) edge[bend left=30] node[above] {$(b^1_2)^{(1)}$} (2) 
      (0) edge[bend right=30] node[below] {$(b^1_3)^{(1)}\hspace{.1in}$} (3) 
  	(1) edge[bend left=30] node[right] {$a^1_2$} (2) 
  	(2) edge[bend left=30] node[left] {$a^2_1$}(1)
 	(2) edge[bend left=30] node[right] {$a^2_3$} (3) 
  	(3) edge[bend left=30] node[left] {$a^3_2$} (2)
 	(1) edge[bend left=85] node[right] {$a^1_3$} (3) 
  	(3) edge[bend left=85] node[left] {$a^3_1$} (1)
      (1) edge[bend left=30] node[above] {$\hspace{.1in}(b^1_1)^{(2)}$} (4) 
      (2) edge[bend left=30] node[above] {$(b^2_1)^{(2)}$} (4) 
      (3) edge[bend right=30] node[below] {$(b^3_1)^{(2)}$} (4) ;
   \end{tikzpicture}}}
\]
\medskip

\[
\begin{split}
x\cdot (b^{1}_j)^{(1)} &=
\gamma\omega^j (b^{1}_j)^{(1)} 
+\lambda \omega^{j-1}  (b^{1}_{j+1})^{(1)}\\
x\cdot (b^i_{1})^{(2)} &=
 - \gamma\omega^{i+1} (b^{i+1}_{1})^{(2)}
+\lambda' \omega^{i-1} (b^i_{ 1})^{(2)}\\
x\cdot (a^i_j)^{(3)} &= \gamma\left(\omega^j (a^i_j)^{(3)} - \omega^{i+1} (a^{i+1}_{j+1})^{(3)}\right) + \lambda_{i, j}^{(3)} (a^i_{j+1})^{(3)}.
\end{split}
\]

\medskip

\noindent where $\lambda:=\lambda_{1,1}^{(1)}$ and $\lambda':=\lambda_{1,1}^{(2)}$. Moreover, let $\lambda'':= \lambda_{1,2}^{(3)}$ and $\lambda''':=\lambda_{1,3}^{(3)}$ for component (3). Putting this all together, the action of $\mathbb{Z}_3$ on $Q$ given by {\textsc{Step 1}}
extends to an action of $T(3)$ on $\kk Q$,  as follows:

\[
\begin{array}{|lllll|}
\hline
x \cdot v_1 = \gamma \omega(v_1 - \omega v_2)
&&x \cdot v_2 =\gamma \omega^2(v_2 - \omega v_3)
&&x \cdot v_3 =\gamma (v_3 - \omega v_1)\\
x \cdot v_0 =0 && x \cdot v_4 =0 &&\\
\hline
\end{array}
\]

\[
\begin{array}{|lccl|}
\hline
x \cdot f_1 = \gamma \omega f_1 {+ \zl} f_2&&& 
x \cdot f_4 = {\zl'} f_4 - \gamma \omega^2 f_5 \\
x \cdot f_2 =\gamma \omega^2 f_2 {+ \zl \omega} f_3 &&& 
x \cdot f_5 ={\zl' \omega} f_5 - \gamma f_6 \\
x \cdot f_3 = \gamma  f_3 {+ \zl  \omega^2} f_1 &&& 
x \cdot f_6 = {\zl'  \omega^2} f_6 - \gamma \omega f_4 \\
x \cdot f_7 = \gamma\left(\omega^2 f_7 - \omega^2 f_8 \right) + {\lambda''} f_{10}&&& 
x \cdot f_{10} =  \gamma\left(f_{10} - \omega^2 f_{12} \right) + {\lambda'''} f_7\\
x \cdot f_8 = \gamma\left(f_8 -  f_9 \right) + {\lambda''} \omega f_{12}&&& 
x \cdot f_{11} = \gamma\left(\omega^2 f_{11} - \omega f_{10} \right) + {\lambda'''} \omega f_9\\
x \cdot f_9 =\gamma\left(\omega f_9 - \omega f_7 \right) + {\lambda''} \omega^2 f_{11}&&& 
x \cdot f_{12} = \gamma\left(\omega f_{12} -  f_{11} \right) + {\lambda'''} \omega^2 f_8. \\
\hline
\end{array}
\]
\[
\begin{array}{|c|}
\hline
\text{for any scalars }\gamma, \zl, \zl', \zl'', \zl''' \in \kk\\
\hline
\end{array}
\]
\end{example}

\bigskip

\section{Extended actions of other pointed Hopf algebras on path algebras}\label{sec:Uq}
In this section, we extend the results in the previous sections to actions of the Frobenius-Lusztig kernel [Section~\ref{sec:uqsl}] and to actions of the Drinfeld double of a Taft algebra [Section~\ref{sec:double}].  Consider the following notation.

\begin{notation} ($q$, $T(n, \xi)$). \label{not:q} Let $q \in \kk$ be a primitive $2n$-th root of unity, and let $\xi \in \kk$ be a primitive $n$-th root of unity. Moreover, let $T(n, \xi)$ be the Taft algebra generated by a grouplike element $g$ and a $(1,g)$-skew primitive element $x$, subject to the relations $g^n=1$, $x^n=0$, $xg=\xi gx$. Note that $T(n,\zeta) = T(n)$ as in Definition~\ref{def:Taft}, for $\zeta$ the fixed primitive $n$-th root of unity of this work.
\end{notation}

\subsection{Actions of $\uqsl$} \label{sec:uqsl}

Consider the following definition.

\begin{definition}[$u_q(\mathfrak{sl}_2)$, Borel subalgebras $\uqslu$ and $\uqsll$]
The {\it Frobenius-Lustzig kernel} $u_q(\mathfrak{sl}_2)$ is the Hopf algebra generated by a grouplike element $K$, a  $(1, K)$-skew-primitive element $E$, and a $(K^{-1}, 1)$-skew-primitive element
$F$, subject to the relations
$$
KE = q^2 EK, \quad KF = q^{-2} FK, \quad K^n=1, \quad E^n=F^n=0,
$$
\begin{equation}\label{eq:efrelation}
EF - FE = \frac{K-K^{-1}}{q- q^{-1}}.
\end{equation}
Note that $\uqsl$ is pointed, of dimension $n^3$. Let $\uqslu$ be the Hopf subalgebra of $\uqsl$ generated by $K, E$, and let $\uqsll$ be the Hopf subalgebra of $\uqsl$ generated by $K, F$; we refer to these as \emph{Borel subalgebras}.
\end{definition}

\begin{lemma}\label{lem:borels}
There are isomorphisms of Hopf algebras
\begin{equation}
\uqslu \simeq T(n, q^{-2}) \quad \quad \text{and} \quad \quad \uqsll \simeq  T(n, q^2).
\end{equation}
\end{lemma}
\begin{proof}
This is easy to check: identify $K$ with $g$ for both $\uqslu$ and $\uqsll$,  and $E$ with $x$ for  $\uqslu$, and $F$ with $g^{-1}x$ for $\uqsll$. 
\end{proof}

Since $\uqsl$ is generated by the Hopf subalgebras $\uqslu$ and $\uqsll$, an action of $\uqsl$ on a path algebra $\kk Q$ is determined by its restriction to these Hopf subalgebras. 
Since these actions are given by the results on Taft actions on $\kk Q$ in the previous sections, we can classify actions of $\uqsl$ by determining some additional compatibility conditions on the scalar parameters from these Taft actions. 
We begin with the result below.

\begin{proposition} \label{prop:UQ0} Let $\uqsl$ act on the path algebra of a set of vertices $\{1, \dotsc, m\}$, labeled so that $K \cdot e_i = e_{i+1}$ with subscripts taken modulo $m$.
\begin{enumerate}[(i)]
\item If $m < n$, then $m=1$ or $m=2$ and the action factors through the quotient $\uqsl/\langle E,\, F \rangle$;
\item If $m = n$, then
\begin{equation} \label{eq:Uqei}
E\cdot e_i = \gamma^E q^{-2i}( e_{i} - q^{-2} e_{i+1}) , \quad \quad \quad
F \cdot e_i = \gamma^F q^{2i} ( e_{i-1} - q^2 e_{i}),
\end{equation}
for some $\gamma^E, \gamma^F \in \kk$. 
If furthermore $n \geq 3$, then the action is subject to the restriction
\begin{equation} \label{eq:Ueicond}
-\gamma^E \gamma^F q^{-1}(q^2-1)^2 ~=~ 1.
\end{equation}
\end{enumerate}
\end{proposition}
\begin{proof}
Part (i) is a result of Proposition~\ref{prop:TaftactQ0}, translated through the isomorphisms of Lemma \ref{lem:borels}.  The condition $m=1$ or $m=2$ comes from the relation \eqref{eq:efrelation}, which implies that $K-K^{-1}$ acts by zero in this case. 
For (ii), Proposition~\ref{prop:TaftactQ0} also gives the formulas \eqref{eq:Uqei}, and the first row of relations in the definition of $u_q(\mathfrak{sl}_2)$ then hold. On the other hand, substituting~\eqref{eq:Uqei}  into the relation \eqref{eq:efrelation} of $u_q(\mathfrak{sl}_2)$ yields
\begin{equation*}
(EF - FE)\cdot e_i = \gamma^E \gamma^F (q^2-1)(e_{i-1} - e_{i+1}) \quad \quad 
\text{while} \quad \quad
(K-K^{-1})\cdot e_i = e_{i+1} - e_{i-1}.
\end{equation*}
If $n=m=2$, then $e_{i-1} = e_{i+1}$ and the relation~\eqref{eq:efrelation} imposes no restriction on the action \eqref{eq:Uqei}.
If $n=m\geq 3$, then $e_{i+1} \neq e_{i-1}$ and the relation~\eqref{eq:efrelation} imposes the restriction \eqref{eq:Ueicond} on the action \eqref{eq:Uqei}.
\end{proof}

We get the following immediate consequence.

\begin{corollary}
If $n \geq 3$, then either one of $\gamma^E$ or $\gamma^F$ determines the other. So, an action of $\uqsl$ on a $\mathbb{Z}_n$-orbit of  vertices is completely determined by the action of a Borel subalgebra, $\uqslu$ or $\uqsll$. \qed
\end{corollary}

We now consider $\uqsl$-actions on Type A and Type B $\mathbb{Z}_n$-minimal quivers.  Note that $\uqsl$ can only act on subquivers of $K_m$ and $K_{m, m'}$ for $m, m' \in \{1, 2, n\}$, by Proposition \ref{prop:UQ0}.  We only give theorems describing the case $m=m'=n \geq 3$ here, but the other cases can be worked out similarly.

\begin{theorem}\label{thm:Uqslaction}
Retain the notation of Section \ref{sec:minimalQ} with $K$ in place of $g$, so the action of $K$ is fixed. Namely, $K \cdot a^i_j = \mu_{i,j} a^{i+1}_{j+1}$ for Type A and $K \cdot b^i_j = \mu_{i,j} b^{i+1}_{j+1}$ for Type B. Assume $n \geq 3$. Then any $\uqsl$-action on the path algebra of a $\mathbb{Z}_n$-minimal subquiver $Q$ of $K_n$ is given by 
\begin{equation}\label{eq:Uqa}
E\cdot a^i_j = \gamma^E(q^{-2j}a^{i}_{j} - q^{-2i-2}a^{i+1}_{j+1}) + \lambda^E_{i,j} a^{i}_{j+1}, \quad \quad
F \cdot a^i_j =\gamma^F(q^{2j} a^{i-1}_{j-1} - q^{2i+2} a^i_j) + \lambda^F_{i,j}a^{i-1}_j,
\end{equation}
subject to the restriction \eqref{eq:Ueicond} along with:
\begin{itemize}
\item $\lambda^E_{i,j} = 0 ,~~\text{ if }i=j \text{ or the arrow } a^i_{j+1} \text{ does not exist in }Q$;
\smallskip
\item $\lambda^F_{i,j} = 0, ~~ \text{ if }i=j \text{ or the arrow } a^{i-1}_{j} \text{ does not exist in }Q$;
\smallskip
\item $\lambda^E_{i+1,j+1} = q^{-2} \lambda^E_{i,j}, \quad \lambda^F_{i+1,j+1} = q^{2} \lambda^F_{i,j}, \quad \text{and}\quad \lambda^F_{i,j}\lambda^E_{i-1,j}  = \lambda^E_{i,j} \lambda^F_{i,j+1}.$
\end{itemize}

\medskip
\noindent Any $u_q(\mathfrak{sl}_2)$-action on the path algebra of a $\mathbb{Z}_n$-minimal subquiver $Q$ of $K_{n,n}$ is given by
\begin{equation}
E\cdot b^i_j =
(\gamma^E_-) q^{-2j}b^{i}_{j} - (\gamma^E_+) q^{-2i-2}b^{i+1}_{j+1} + \lambda^E_{i,j} b^{i}_{j+1}, \quad \quad
F \cdot a^i_j =
(\gamma^F_-) q^{2j} b^{i-1}_{j-1} - (\gamma^F_+) q^{2i+2} b^i_j + \lambda^F_{i,j}b^{i-1}_j,
\end{equation}
subject to restrictions of the form \eqref{eq:Ueicond} on $\gamma_{\pm}^E$ and $\gamma_{\pm}^F$, along with
\begin{itemize}
\item $\lambda^E_{i, j} = 0, ~~$ if $i=j$ or the arrow $b^i_{j+1}$ does not exist in $Q$;
\smallskip
\item $\lambda^F_{i,j} = 0, ~~$ if $i=j$ or the arrow $b^{i-1}_j$ does not exist in $Q$;
\smallskip
\item $\lambda^E_{i+1,j+1} = q^{-2} \lambda^E_{i,j}, \quad \lambda^F_{i+1,j+1} = q^{2} \lambda^F_{i,j}, \quad \text{and} \quad \lambda^F_{i,j}\lambda^E_{i-1,j}  = \lambda^E_{i,j} \lambda^F_{i,j+1}$;
\smallskip
\item $(\gamma^E_+)^n = (\gamma^E_-)^n + \prod_{\ell=0}^{n-1} \lambda^E_{i, j+\ell}$ \quad and \quad $(\gamma^F_+)^n = (\gamma^F_-)^n + \prod_{\ell=0}^{n-1} \lambda^F_{i, j+\ell}$.
\end{itemize}
\end{theorem}

\begin{proof}
First consider the Type A case. Proposition~\ref{prop:UQ0} gives the $\uqsl$-action on the vertices, imposing the restriction~\eqref{eq:Ueicond}. Theorem \ref{thm:T(n)typeA}, translated through the isomorphisms of Lemma~\ref{lem:borels}, initially gives expressions for $E\cdot a^i_j$ and $F\cdot a^i_j$ which involve parameters $\mu_{i,j}$. However, we examine the coefficients of various arrows in the relation \eqref{eq:efrelation} applied to $a^i_j$, and find that the coefficient of $a^{i-1}_{j-1}$ gives
\begin{equation} \label{eq:ai-1,j-1}
\gamma^E \gamma^F(q^2-1) + \mu_{i,j}(q-q^{-1})^{-1} =0. 
\end{equation}
By applying \eqref{eq:Ueicond}, this simplifies to $(\mu_{i,j}-1)(q-q^{-1})^{-1} = 0$. So, $\mu_{i,j} = 1$ for all $i, j$, which gives our formulas \eqref{eq:Uqa}.
Similarly, the coefficient of $a^{i-1}_{j+1}$ in the relation \eqref{eq:efrelation} applied to $a^i_j$ implies that
\begin{equation}
\lambda^F_{i,j}\lambda^E_{i-1,j}  - \lambda^E_{i,j} \lambda^F_{i,j+1} = 0.
\end{equation}
Theorem \ref{thm:T(n)typeA} gives the remaining restrictions on the scalars, and the coefficients of other arrows yield restrictions that are already implied by those above.

For Type B, the action is derived exactly as in the Type A case, by replacing $a^{\star}_{\star}$ with $b^{\star}_{\star}$, and $\gamma$ with $\gamma_{\pm}$, appropriately.
\end{proof}

\subsection{Actions of the double $D(T(n))$} \label{sec:double}
We take the presentation of the Drinfeld double of the $n$-th Taft algebra in \cite[Definition-Theorem~3.1]{Chen}. Namely, by \cite[Theorem~3.3]{Chen}, we have that the double is isomorphic to the Hopf algebra $H_n(p,q)$ generated by $a,b,c,d$. We take $p=1, ~q=\zeta^{-1}, ~a=x, ~b=g, ~c=G, ~d=X$ to get that: 

\begin{definition}[$D(T(n))$] The {\it Drinfeld double  $D(T(n))$ of the $n$-th Taft algebra} is generated by $g, x, G, X$, subject to relations:
\[
xg = \zeta gx, \quad GX = \zeta XG,  \quad gX=\zeta Xg,  \quad 
xG=\zeta Gx, \quad gG= Gg,   \quad g^n = G^n = 1,  \quad  x^n = X^n = 0,
\]
\begin{equation} \label{eq:xX}
xX-\zeta Xx = \zeta(gG-1).
\end{equation}

Here, $g$ and $G$ are grouplike elements, $x$ is $(1,g)$-a skew primitive element, and $X$ is a $(1,G)$-skew primitive element. Here, $g, x$ generate the Hopf subalgebra $T(n)$, and $G=g^*, X=x^*$ generate $(T(n)^{op})^* = (T(n)^*)^{cop}$.
\end{definition}

The following lemma is easy to check.

\begin{lemma}\label{lem:Dsubalg}
The Hopf subalgebra of $D(T(n))$ generated by $g, x$ is isomorphic to $T(n)=T(n,\zeta)$, as Hopf algebras. Moreover, the Hopf subalgebra of $D(T(n))$ generated by $G, X$ is isomorphic to $T(n,\zeta^{-1})$, as Hopf algebras. \qed
\end{lemma}

To give the complete classification of $D(T(n))$-actions on path algebras of quivers, one would need to define {\it $\mathbb{Z}_n \times \mathbb{Z}_n$-minimal quivers and consider the $D(T(n))$-action on these}, since $\mathbb{Z}_n \times \mathbb{Z}_n$ is isomorphic to the group of grouplike elements of $D(T(n))$. This is the subject of future work. Our aim for now is to exhibit an action of $D(T(n))$ on $\kk Q$, where $Q$ admits an action of $\mathbb{Z}_n$. We begin by providing an action of $D(T(n))$ on  a $\mathbb{Z}_n$-orbit of vertices. 

\begin{proposition} \label{prop:DQ0} An action of $D(T(n))$ on $\kk Q_0$ where $Q_0 = \{1, \dotsc, n\}$ is given by 
\begin{equation} \label{eq:Dei}
g \cdot e_i = e_{i+1}, \quad \quad G \cdot e_i = e_{i+1}, \quad \quad
x \cdot e_i = \gamma^x \zeta^i(e_i - \zeta e_{i+1}), \quad \quad 
X \cdot e_i = \gamma^X \zeta^{-i}(e_i - \zeta^{-1} e_{i+1}).
\end{equation}
for some  $\gamma^x, \gamma^X \in \kk$.
If $n \geq 3$, then these scalars are subject to the restriction
\begin{equation} \label{eq:Deicond}
\gamma^x \gamma^X(1- \zeta^{-1}) = 1.
\end{equation}
\end{proposition}

\begin{proof}
The formulas \eqref{eq:Dei} satisfy the first row of relations of $D(T(n))$. On the other hand, the restriction~\eqref{eq:Deicond} comes from substituting~\eqref{eq:Dei}  into the relation \eqref{eq:xX} of $D(T(n))$ when $n \geq 3$. If $n=2$, the restriction from~\eqref{eq:xX} is vacuous. 
\end{proof}

We proceed by extending the $D(T(n))$-action on vertices in Proposition~\ref{prop:DQ0} to yield an action of $D(T(n))$ on Type A and Type B $\mathbb{Z}_n$-minimal quivers.  As in the $\uqsl$ case, we only give theorems describing the cases when each orbit of vertices has $n$ elements here.

\begin{theorem}\label{thm:Daction}
Retain the notation of Section~\ref{sec:minimalQ}.
An action of $D(T(n))$ on the path algebra of a Type $A$ $\mathbb{Z}_n$-minimal subquiver $Q$ of $K_n$ is given by 
\begin{equation} \label{eq:DA}
\begin{array}{rl}
g \cdot a^i_j = \mu^g_{i,j} a^{i+1}_{j+1},  ~~&~~   G \cdot a^i_j = \mu^G_{i,j} a^{i+1}_{j+1},\\\\
x \cdot a^i_j = \gamma^x\left(\zeta^j a^i_j - \zeta^{i+1} \mu^g_{i,j} a^{i+1}_{j+1}\right) + \lambda^x_{i,j} a^i_{j+1},  ~~&  ~~
X \cdot a^i_j = \gamma^X\left(\zeta^{-j} a^i_j - \zeta^{-i-1} \mu^G_{i,j} a^{i+1}_{j+1}\right) + \lambda^X_{i,j} a^i_{j+1},\\
\end{array}
\end{equation}
subject to the restrictions:
{\small $$\mu_{i,i}^g = \mu_{i,i}^G =1 \text{ for all }i, \quad \quad 
\textstyle \prod_{\ell=0}^{n-1}\mu_{i+\ell,j+\ell}^g = \prod_{\ell=0}^{n-1}\mu_{i+\ell,j+\ell}^G = 1, 
\quad \quad  \mu_{i,j}^G \mu_{i+1,j+1}^g = \mu_{i,j}^g \mu_{i+1,j+1}^G,$$}

 \vspace{-.3in}
 
{\small $$\zeta \mu_{i,j+1}^g \lambda_{i,j}^x = \mu_{i,j}^g \lambda_{i+1,j+1}^x,
\quad \zeta \mu_{i,j+1}^G \lambda_{i,j}^X = \mu_{i,j}^G \lambda_{i+1,j+1}^X, \quad  \zeta \mu_{i,j}^g \lambda_{i+1,j+1}^X =  \mu_{i,j+1}^g \lambda_{i,j}^X, 
\quad 
\zeta \mu_{i, j+1}^G \lambda_{i,j}^x = \mu_{i,j}^G \lambda_{i+1,j+1}^x,$$ }

 \vspace{-.3in}
 
{\small $$\lambda_{i,j}^X \lambda_{i,j+1}^x - \zeta \lambda_{i,j}^x \lambda_{i,j+1}^X = 0, \quad  \quad \lambda_{i,j}^x = \lambda_{i,j}^X =0 \text{ if either } i=j, \text{ or the arrow } a^i_{j+1} \text{ does not exist in } Q.$$}

 \vspace{-.1in}
 
\noindent If $n \geq 3$, then we also impose the condition: $\gamma^x \gamma^X(1- \zeta^{-1}) = 1$.
\medskip

An action of $D(T(n))$ on the path algebra of a Type B minimal subquiver $Q$ of $K_{n,n}$ is given by
\begin{equation} \label{eq:DB}
\begin{array}{rl}
g \cdot b^i_j = \mu^g_{i,j} b^{i+1}_{j+1}, ~&~ G \cdot b^i_j = \mu^G_{i,j} b^{i+1}_{j+1}\\\\
x \cdot b^i_j = (\gamma_-^x) \zeta^j b^i_j - (\gamma_+^x) \zeta^{i+1} \mu^g_{i,j} b^{i+1}_{j+1} + \lambda^x_{i,j} b^i_{j+1}, ~&~
X \cdot b^i_j = (\gamma_-^X) \zeta^{-j} b^i_j - (\gamma_+^X) \zeta^{-i-1} \mu^G_{i,j} b^{i+1}_{j+1} + \lambda^X_{i,j} b^i_{j+1},\\
\end{array}
\end{equation}
subject to the restrictions:
{\small $$\mu_{i,i}^g = \mu_{i,i}^G =1 \text{ for all }i, \quad 
\textstyle \prod_{\ell=0}^{n-1}\mu_{i+\ell,j+\ell}^g = \prod_{\ell=0}^{n-1}\mu_{i+\ell,j+\ell}^G = 1, 
 \quad \mu_{i,j}^G \mu_{i+1,j+1}^g = \mu_{i,j}^g \mu_{i+1,j+1}^G,$$}
 
 \vspace{-.3in}
 
{\small $$\zeta \mu_{i,j+1}^g \lambda_{i,j}^x = \mu_{i,j}^g \lambda_{i+1,j+1}^x,
\quad 
\zeta \mu_{i,j+1}^G \lambda_{i,j}^X = \mu_{i,j}^G \lambda_{i+1,j+1}^X, \quad
\zeta \mu_{i,j}^g \lambda_{i+1,j+1}^X =  \mu_{i,j+1}^g \lambda_{i,j}^X, 
\quad \
\zeta \mu_{i, j+1}^G \lambda_{i,j}^x = \mu_{i,j}^G \lambda_{i+1,j+1}^x,$$}

 \vspace{-.3in}

{\small $$\lambda_{i,j}^X \lambda_{i,j+1}^x - \zeta \lambda_{i,j}^x \lambda_{i,j+1}^X = 0, \quad\lambda_{i,j}^x = \lambda_{i,j}^X =0 \text{ if either } i=j \text{ or the arrow } a^i_{j+1} \text{ does not exist in } Q$$}

 \vspace{-.1in}

\noindent If $n \geq 3$, then we also impose the condition: $(\gamma_\pm^x) (\gamma_\pm^X)(1- \zeta^{-1}) = 1.$
\end{theorem}

\begin{proof}
First, Proposition~\ref{prop:DQ0} gives an $D(T(n))$-action on a $\mathbb{Z}_n$-orbit of vertices; this imposes the restriction~\eqref{eq:Deicond} when $n \geq 3$. Now, consider the Type A case. Theorem \ref{thm:T(n)typeA} gives the expressions in \eqref{eq:DA}, along with some restrictions on parameters, which are listed in the first row of restrictions in the statement of the theorem. Now we need to check that the expressions, $g \cdot a^i_j$, $G \cdot a^i_j$, $x \cdot a^i_j$, $X \cdot a^i_j$, satisfy the relations of $D(T(n))$.  The relations $gG = Gg$,  $gX=\zeta Xg$, $xG = \zeta Gx$,  imply, respectively, that
$$\mu_{i,j}^G \mu_{i+1,j+1}^g = \mu_{i,j}^g \mu_{i+1,j+1}^G, \quad \quad 
 \zeta \mu_{i,j}^g \lambda_{i+1,j+1}^X =  \mu_{i,j+1}^g \lambda_{i,j}^X, \quad \quad
\zeta \mu_{i, j+1}^G \lambda_{i,j}^x = \mu_{i,j}^G \lambda_{i+1,j+1}^x.
$$
With these restrictions, it is straight-forward to check that the relation $xX - \zeta Xx-\zeta(gG-1) =0$ yields $$\lambda_{i,j}^X \lambda_{i,j+1}^x - \zeta \lambda_{i,j}^x \lambda_{i,j+1}^X = 0,$$
and we are done with the Type A case.
For Type B, the action is derived exactly as in the Type A case, by replacing $a^{\star}_{\star}$ with $b^{\star}_{\star}$, and $\gamma$ with $\gamma_{\pm}$, appropriately.
\end{proof}

\subsection{Gluing} \label{sec:glueUqD}
Gluing is achieved in the following manner.
Let $Q$ be a quiver that admits an action of $\mathbb{Z}_n$.
Since the actions of $\uqsl$ and $D(T(n)))$ on $\kk Q$ in Theorems~\ref{thm:Uqslaction} and~\ref{thm:Daction} are  determined by Taft actions, Theorem~\ref{thm:glue} and the algorithm in Section~\ref{sec:algorithm} applies. In other words, we can glue the actions of $\uqsl$ or of $D(T(n))$ on path algebras of $\mathbb{Z}_n$-minimal quivers. Hence, we obtain an action of $\uqsl$ (resp., $D(T(n))$) extending a Taft action on any path algebra $\kk Q$, where each path algebra of a $\mathbb{Z}_n$-minimal component of $Q$ admits an action of $\uqsl$ (resp., $D(T(n))$).


\section{Appendix} \label{sec:appendix}

In this appendix, we show that the relation $x^n=0$ imposes no restrictions on the $x$-action on the vertex $e_i$ and on the arrows $a^i_j$, $b^i_j$ of $\kk Q$, which is used in the proofs of Proposition~\ref{prop:TaftactQ0}, Theorems~\ref{thm:T(n)typeA} and~\ref{thm:T(n)typeB}.  Recall that $\zeta$ is a primitive $n$-th root of unity, with $n \geq 2$.  An easy computation shows that
\begin{equation}\label{eq:prodzeta}
\prod_{\ell = 1}^n \zeta^{\ell} = \zeta^{\frac{n(n+1)}{2}} =
\begin{cases}
1 & n\ \text{odd}\\
-1 &n\ \text{even}
\end{cases},
\end{equation}
a fact we will use several times here.

This appendix uses some basic properties of {symmetric polynomials} and $q$-binomial coefficients, which we recall here.
For integers $a,b$, the \emph{complete symmetric polynomial of degree $a$}, denoted $h_a(x_0, \dotsc, x_b)$, is defined as the sum of all monomials of total degree $a$ in the variables $x_0, \dotsc, x_b$.  We interpret the variable set as being empty when $b< 0$.  When $a=0$, we have that $h_0(x_0, \dotsc, x_b) = 1$ for any integer $b$. When $a \neq 0$, we have that $h_a(x_0, \dotsc, x_b) = 0$ if $a<0$ or $b<0$.  Also, a product $\prod_{i=a}^b$ is taken to be 1 whenever $b<a$. 
These polynomials are homogeneous, meaning that $h_a(cx_0, \dotsc, cx_b) = c^a h_a(x_0, \dotsc, x_b)$ for any scalar $c$.  We have immediately from the definition that
\begin{equation}\label{eq:hidentity}
h_{a}(x_0, \dotsc, x_{b-1}) + x_b h_{a-1}(x_0, \dotsc, x_b) = h_{a}(x_0, \dotsc, x_b)
\end{equation}
for any integers $a, b$.

\begin{lemma}\label{lem:hvanish}
For nonnegative integers $a, b$, we have that $h_a (1, \zeta, \zeta^2, \dotsc, \zeta^{b}) = 0$ when $n$ divides $a+b$ and $a, b \neq 0$.
\end{lemma}

\begin{proof}
Recall that for a nonnegative integer $k$, the corresponding $q$-integer is defined as $[k]_q = 1 + q + q^2 + \cdots +q^{k-1}$, and that its evaluation $[k]_{q=\zeta}$ is 0 if and only if $n$ divides $k$.  The $q$-binomial coefficient ${i \brack j}_q$ is defined by analogous factorial formula for binomial coefficients, so we have that the evaluation ${i \brack j}_{q=\zeta}$ vanishes when $n$ divides $i$ and $j \neq 0,\, i$.

The {\it principal specialization} from \cite[Proposition~7.8.3]{Stanley}, with $q= \zeta$, gives us that
\[
h_a (1, \zeta, \zeta^2, \dotsc, \zeta^{b}) = {a+b \brack a}_{q=\zeta},
\]
so by the paragraph above this quantity vanishes when $n$ divides $a+b$ and $a, b\neq 0$.
\end{proof}

We prove a general lemma about a linear operator $X$ acting as the element $x$ does.

\begin{lemma}\label{lem:powerx}
Let $V$ be a vector space, and consider a collection of vectors $\{v^i_j\}$ in $V$, where $1 \leq i \leq m$ and $1\leq j \leq m'$. When $i$ (resp. $j$) is outside this range, interpret it modulo $m$ (resp., $m'$) to lie in this range. Let $X$ be a linear operator acting on $V$ such that
\[
Xv^i_j = \eta_j v^i_j + \theta_{i,j} v^{i+1}_{j+1} + \tau_{i,j} v^i_{j+1}
\]
for some scalars $\eta_j,\, \theta_{i,j},\, \tau_{i,j} \in \kk$.  Furthermore, assume that the scalars satisfy the relation 
\begin{equation}\label{eq:tautheta}
\tau_{i+1, j+1}\theta_{i,j} = \zeta \theta_{i,j+1}\tau_{i,j}
\end{equation}
for all pairs $(i, j)$.  Then, for all positive integers $k$, we have that
\[
X^k (v^i_j) = \sum_{0 \leq s \leq t \leq k} \psi_{k,s,t} v^{i+s}_{j+t}
\]
where
\begin{equation}\label{eq:psivalue}
\psi_{k,s,t} = \left(\prod_{\ell=0}^{s-1} \theta_{i+\ell, j+\ell +t-s}\right) \left(\prod_{\ell =0}^{t-s-1} \tau_{i,j+\ell} \right)  h_{k-t}(\eta_j, \eta_{j+1}, \cdots, \eta_{j+t}) h_s(1, \zeta, \dots, \zeta^{t-s}).
\end{equation}
\end{lemma}

\begin{proof}
We proceed by induction on $k$. The base case of $k=1$ follows from simple substitution.  So assume that the statement is true when $k$ is replaced by $k-1$; that is, we have a formula for $X^{k-1} (v^i_j)$ for all pairs $(i, j)$.  Now, when applying $X$ to $X^{k-1} (v^i_j)$, the coefficient $\psi_{k,s,t}$ is the sum of contributions from three terms.  Namely, we have that
\begin{itemize}
\item $X (\psi_{k-1, s, t} v^{i+s}_{j+t})$ contributes $\eta_{j+t} \psi_{k-1, s, t}$,
\item $X (\psi_{k-1, s-1, t-1} v^{i+s-1}_{j+t-1})$ contributes $\theta_{i+s-1, j+t-1} \psi_{k-1, s-1, t-1}$,
\item $X (\psi_{k-1, s, t-1} v^{i+s}_{j+t-1})$ contributes $\tau_{i+s, j+t-1} \psi_{k-1, s, t-1}$.
\end{itemize}
So we obtain
\begin{equation}\label{eq:blah}
\begin{split}
\psi_{k,s,t} = \left(\prod_{\ell=0}^{s-1} \theta_{i+\ell, j+\ell+t-s}\right) \left(\prod_{\ell =0}^{t-s-1} \tau_{i,j+\ell} \right) 
\Bigl[\ &\eta_{j+t}h_{k-1-t}(\eta_j, \dotsc, \eta_{j+t})  h_s(1, \zeta, \dotsc, \zeta^{t-s}) \\
+ &h_{k-t} (\eta_j, \dotsc, \eta_{j+t-1}) h_{s-1}(1, \zeta, \dotsc, \zeta^{t-s})\Bigr]
\end{split}
\end{equation}
\vspace{-.3in}

$$\hspace{.6in} +\tau_{i+s,j+t-1}\left( \prod_{\ell=0}^{s-1}\theta_{i+\ell, j+\ell+t-s-1}\right)\left( \prod_{\ell=0}^{t-s-2} \tau_{i,j+\ell}\right)h_{k-t}(\eta_j, \dots, \eta_{j+t-1})h_s(1, \zeta, \dots, \zeta^{t-1-s}).$$

\noindent To simplify the last summand, we use \eqref{eq:tautheta} to compute
\begin{equation}\label{eq:taupasstheta}
\begin{split}
\tau_{i+s, j+t-1} \textstyle \prod_{\ell=0}^{s-1} \theta_{i+\ell, j+\ell +t-1-s} =\ &\left(\tau_{i+s, j+t-1}\right) \theta_{i+s-1, j+t-2} \theta_{i+s-2, j+t-3} \cdots \theta_{i,j+t-1-s} \\
=\ &\zeta \theta_{i+s-1, j+t-1} \left(\tau_{i+s-1, j+t-2}\right)\theta_{i+s-2, j+t-3} \cdots   \theta_{i,j+t-1-s}\\
=\ &\zeta^2 \theta_{i+s-1, j+t-1} \theta_{i+s-2, j+t-2} \left(\tau_{i+s-2, j+t-3}\right)\cdots  \theta_{i,j+t-1-s} \\
\dots =\ &\zeta^s \left(\tau_{i, j+t-s -1}\right) \textstyle \prod_{\ell=0}^{s-1} \theta_{i+\ell, j+\ell +t-s}.
\end{split}
\end{equation}

\noindent By \eqref{eq:blah} and \eqref{eq:taupasstheta}, we now have that
\begin{equation}\label{eq:psisum1}
\begin{split}
\psi_{k,s,t} = \left(\prod_{\ell=0}^{s-1} \theta_{i+\ell, j+\ell+t-s}\right) \left(\prod_{\ell =0}^{t-s-1} \tau_{i,j+\ell} \right) 
\Bigl[\ &\eta_{j+t}h_{k-1-t}(\eta_j, \dotsc, \eta_{j+t})  h_s(1, \zeta, \dotsc, \zeta^{t-s}) \\
+ &h_{k-t} (\eta_j, \dotsc, \eta_{j+t-1}) h_{s-1}(1, \zeta, \dotsc, \zeta^{t-s})\\
+ & h_{k-t}(\eta_j, \dotsc, \eta_{j+t-1}) \, \zeta^s \, h_s (1, \zeta, \dotsc, \zeta^{t-s-1})\Bigr] .
\end{split}
\end{equation}
The factor of $\zeta^s$ in the last term can be absorbed into the (homogeneous) polynomial $h_s$ to make the sum of the last two terms equal to
\[
h_{k-t} (\eta_j, \dotsc, \eta_{j+t-1}) \left( h_{s-1}(1, \zeta, \dotsc, \zeta^{t-s}) + h_s (\zeta, \zeta^2, \dotsc, \zeta^{t-s}) \right).
\]
Now, taking $x_{t-s-i} = \zeta^i$, \eqref{eq:hidentity} implies  that the two terms in the parenthesis simplify to 
$
 h_s(1, \zeta, \dotsc, \zeta^{t-s})$.
Thus, \eqref{eq:psisum1} yields
\[
\psi_{k,s,t} = \left(\prod_{\ell=0}^{s-1} \theta_{i+\ell,j+\ell+t-s}\right) \left(\prod_{\ell =0}^{t-s-1} \tau_{i,j+\ell} \right)  h_s(1, \zeta, \dotsc, \zeta^{t-s}) [ \eta_{j+t}h_{k-1-t}(\eta_j, \dotsc, \eta_{j+t}) + h_{k-t}(\eta_j, \dotsc, \eta_{j+t-1}) ].
\]
Again applying \eqref{eq:hidentity}, we find that \eqref{eq:psivalue} holds, and the induction is complete.
\end{proof}

We can apply this result to show that $x^n$ acts by 0 in the following cases.

\begin{lemma} \label{lem:x on ei}
In the notation of Proposition \ref{prop:TaftactQ0}, we have that $x^n \cdot e_i =0$.
\end{lemma}
\begin{proof} 
Fix $i$.  We will apply Lemma~\ref{lem:powerx} in the case $k=n,\ m=m'$, with $X=x$ and $v^i_i = e_i$ for $1 \leq i \leq m$. The $v^i_j$ do not play a role in the proof for $i \neq j$, so we set $\theta_{i,j}=\tau_{i,j} = 0$ when $i\neq j$ and  assume $i=j$ for the remainder of the proof.
From the proof of Proposition \ref{prop:TaftactQ0}, we have that $\eta_i=\gamma \zeta^i$, $\theta_{i,i}=-\gamma \zeta^{i+1}$, and $\tau_{i,i} = 0$ for all $i$.  So, \eqref{eq:tautheta} holds and we can apply Lemma~\ref{lem:powerx}. Making these substitutions into \eqref{eq:psivalue}, we immediately see that $\psi_{n, s, t} = 0$ unless $s=t$ by considering the $\tau$ factor.  In case $s=t$, the $\tau$ factor and the last factor are equal to 1. This gives
\[
\psi_{n, s, s} = \left(\prod_{\ell=0}^{s-1} \theta_{i+\ell, i+\ell}\right) h_{n-s}(\eta_i, \eta_{i+1}, \cdots, \eta_{i+s}) .
\]
Since $h_{n-s}(\eta_i, \eta_{i+1}, \cdots, \eta_{i+s})$ is a scalar multiple of $h_{n-s}(1, \zeta, \dotsc, \zeta^s)$, this vanishes by Lemma \ref{lem:hvanish} unless $s=0$ or $s=n$.

So we have that $x^n \cdot e_i = (\psi_{n, 0, 0} + \psi_{n, n, n})e_i$.  These are easily computed by substitution to be
\[
\psi_{n, 0, 0} = h_n (\eta_i) = (\eta_i)^n = \zg^n
\]
\[
\psi_{n, n, n} ~=~ \prod_{\ell = 0}^{n-1} \theta_{i+\ell, i+\ell} ~=~ (-1)^n\zg^n \prod_{\ell = 0}^{n-1} \zeta^{i+\ell+1} ~=~  (-1)^n \zg^n \prod_{\ell = 0}^{n-1} \zeta^{\ell} ~=~ - \zg^n,
\]
where the last equality invokes Equation \eqref{eq:prodzeta}.
Thus we have shown that $x^n \cdot e_i= 0$.
\end{proof}

\begin{lemma}\label{lem:xact0A}
In the notation of Theorem \ref{thm:T(n)typeA}, we have that $x^n \cdot a^i_j = 0$ for all $(i,j)$.
\end{lemma}

\begin{proof}
We will apply Lemma \ref{lem:powerx} in the case $k=n$ with $X=x$ and $v^i_j = a^i_j$.  From the proof of Theorem \ref{thm:T(n)typeA}, we have that $\eta_j=\gamma \zeta^j$, $\theta_{i,j}=-\gamma \mu_{i,j} \zeta^{i+1}$, and $\tau_{i,j} = \lambda_{i, j}$, and that these satisfy \eqref{eq:tautheta}.  Now making these substitutions into \eqref{eq:psivalue}, we see that $\psi_{n, s, t}$ is a scalar times
\[
h_s(1, \zeta, \dotsc, \zeta^{t-s}) h_{n-t}(1, \zeta, \dots, \zeta^t) .
\]
By Lemma \ref{lem:hvanish}, the second factor vanishes except when $t=0$ or $t= n$. Then, again by Lemma~\ref{lem:hvanish}, and then the first factor vanishes except when $s=0$ or $s=t=n$. Therefore, we need only consider the cases $(s, t) =(0, 0),\, (0,n),$ and $(n,n)$, in which these factors both equal 1.  Notice that we have in all of these cases, $a^{i+s}_{j+t} = a^i_j$, because the indices for arrows are taken modulo $n$.  So it suffices to show that $\psi_{n, 0, 0} + \psi_{n, 0,n} + \psi_{n, n, n} = 0$.  Substitute and compute, omitting factors equal to 1, we get:
\[
\psi_{n, 0, 0} = h_n (\eta_j) = (\eta_j)^n = \zg^n
\quad \text{ and } \quad
\psi_{n, 0,n} = \prod_{\ell = 0}^{n-1} \tau_{i, j+\ell} = 0.
\]
The latter follows by the second condition of Theorem~\ref{thm:T(n)typeA}, because there will be a factor with $i =j + \ell$ modulo $n$. Further,
\[
\psi_{n, n, n} ~=~ 
\prod_{\ell = 0}^{n-1} \theta_{i+\ell, j+\ell} ~=~ (-1)^n\zg^n \prod_{\ell = 0}^{n-1} \mu_{i+\ell, j+\ell} \zeta^{i+\ell+1} ~=~  (-1)^n \zg^n \prod_{\ell = 1}^n \zeta^{\ell} ~=~ - \zg^n,
\]
where the penultimate equality uses \eqref{eq:prodmu} and the last equality invokes Equation \eqref{eq:prodzeta}.
Thus we have shown that $x^n \cdot a^i_j = 0$.
\end{proof}

\begin{lemma}\label{lem:xact0B}
In the notation of Theorem \ref{thm:T(n)typeB}, we have that $x^n \cdot b^i_j = 0$ for all $(i,j)$.
\end{lemma}
\begin{proof}
Again, we apply Lemma \ref{lem:powerx} with $k=n$, $X=x$, and $v^i_j = b^i_j$.  The proof is the same as Lemma \ref{lem:xact0A}, except in this case we find that
\[
\psi_{n, 0, 0} = \zg_-^n, \qquad \psi_{n, n, n} =  - \zg_+^n, \qquad \text{and} \qquad
\psi_{n, 0,n} = \prod_{\ell = 0}^{n-1} \tau_{i, j+\ell} = \prod_{\ell = 0}^{n-1} \lambda_{i, j+\ell}.
\]
Thus, we see that $x^n$ acts by 0 if and only if $(\gamma_+)^n = (\gamma_-)^n + \prod_{\ell = 0}^{n-1} \lambda_{i, j+\ell}$, which is the condition assumed on the scalars in Theorem \ref{thm:T(n)typeB}.
\end{proof}

\section*{Acknowledgments}
We thank Steven Sam for directing us to the reference for principal specializations of symmetric polynomials. We also thank Mio Iovanov for helping us with the proof of Lemma \ref{lem:Taftfaithful}, and Susan Montgomery for suggesting the extensions of our results in Section \ref{sec:Uq}.
The second author was supported by the National Science Foundation: NSF-grants DMS-1102548 and DMS-1401207.

\bibliography{TaftActionsOnPathAlgs_biblio}

\end{document}